\documentclass[oneside]{amsart}

\usepackage{amsmath, amsthm, amsfonts}
\usepackage{amssymb}
\usepackage[utf8]{inputenc}
\usepackage{latexsym}
\usepackage{verbatim}
\usepackage{url}
\usepackage{textcomp}
\usepackage[all,cmtip]{xy}
\usepackage{color}
\usepackage[colorlinks = true,
linkcolor = blue,
urlcolor  = blue,
citecolor = blue,
anchorcolor = blue]{hyperref}
\usepackage{tabularx}
\input{xypic}
\usepackage{enumerate}
\usepackage{amssymb}
\usepackage{array}
\usepackage{url}
\usepackage{booktabs}
\usepackage{colortbl}
\usepackage{xcolor}
\usepackage{chngpage}

\DeclareMathAlphabet{\mathfr}{U}{euf}{m}{n}

\newtheorem{theorem}{Theorem}[section]
\newtheorem*{theorem*}{Theorem}

\newtheorem{proposition}[theorem]{Proposition}
\newtheorem*{corollary*}{Corollary}
\newtheorem{corollary}[theorem]{Corollary}

\newtheorem{lemma}[theorem]{Lemma}
\newtheorem*{lemma*}{Lemma}

\theoremstyle{remark}
\newtheorem{remark}[theorem]{Remark}

\theoremstyle{definition}
\newtheorem{algorithm}[theorem]{Algorithm}

\newcommand{\Q}{\mathbb Q}
\newcommand{\Qbar}{{\overline{\mathbb Q}}}
\newcommand{\Zbar}{{\overline{\mathbb Z}}}
\newcommand{\Cbar}{{\overline{C}}}
\newcommand{\sigmabar}{{\overline{\sigma}}}
\newcommand{\alphabar}{{\overline{\alpha}}}
\newcommand{\Gal}{\mathrm{Gal}}

\newcommand{\Z}{\mathbb Z}

\newcommand{\F}{\mathbb F}
\newcommand{\Fbar}{\overline \F}

\newcommand{\AGL}{\mathrm{AGL}}
\newcommand{\GL}{\mathrm{GL}}
\newcommand{\GU}{\mathrm{GU}}

\newcommand{\End}{\operatorname{End}}
\newcommand{\Het}{H^1_\text{\'et}}
\newcommand{\Hcrys}{H^1_\text{crys}}
\newcommand{\HdR}{H^1_\text{dR}}

\newcommand{\Frob}{\operatorname{Frob}}

\newcommand{\im}{\operatorname{Im}}
\newcommand{\pr}{\operatorname{pr}}
\newcommand{\sems}{_{\operatorname{ss}}}
\newcommand{\Aut}{\operatorname{Aut}}
\newcommand{\omegab}{\boldsymbol{\omega}}
\newcommand{\p}{\mathfrak{p}}

\newcommand{\lambdabar}{{\overline \lambda}}
\newcommand{\pbar}{{\overline \p}}

\newcommand{\Jac}{\operatorname{Jac}}
\newcommand{\Tr}{\operatorname{Tr}}
\newcommand{\GSp}{\operatorname{GSp}}
\newcommand{\Sp}{\operatorname{Sp}}
\newcommand{\ST}{\mathrm{ST}}

\newcommand{\USp}{\mathrm{USp}}
\newcommand{\U}{\mathrm{U}}
\newcommand{\cO}{\mathcal{O}}
\newcommand{\cC}{\mathcal{C}}

\newcommand{\Nm}{\operatorname{Nm}}
\newcommand{\q}{\mathfrak{q}}

\newcommand{\fI}{\mathfrak{I}}
\newcommand{\fS}{\mathfrak{S}}
\newcommand{\rk}{\mathrm{rk}}

\newcommand{\fSst}{\mathfrak{S}^*}

\newcommand\numberthis{\addtocounter{equation}{1}\tag{\theequation}}
\numberwithin{equation}{section}

\begin{document}
\title[Computing $L$-polynomials of Picard curves]{Computing $L$-polynomials of Picard curves from Cartier--Manin matrices.\\
\tiny{With an appendix by A.V. Sutherland.}}
\author{Sualeh Asif}

\address{Massachusetts Institute of Technology,
77 Massachusetts Ave., Cambridge, MA 02139, United States}
\email{sualeh@mit.edu}
\urladdr{https://sualehasif.me/}

\author{Francesc Fit\'e}

\address{Department of Mathematics,
Massachusetts Institute of Technology,
77 Massachusetts Ave., Cambridge, MA 02139, United States}
\email{ffite@mit.edu}
\urladdr{https://www-math.mit.edu/~ffite/}

\author{Dylan Pentland}
\address{Massachusetts Institute of Technology,
77 Massachusetts Ave., Cambridge, MA 02139, United States}
\email{dylanp@mit.edu}

\date{\today}

\begin{abstract}
We study the sequence of zeta functions $Z(C_p,T)$ of a generic Picard curve $C:y^3=f(x)$ defined over $\Q$ at primes $p$ of good reduction for~$C$. We define a degree 9 polynomial $\psi_f\in \Q[x]$ such that the splitting field of $\psi_f(x^3/2)$ is the $2$-torsion field of the Jacobian of $C$. We prove that, for all but a density zero subset of primes, the zeta function $Z(C_p,T)$ is uniquely determined by the Cartier--Manin matrix $A_p$ of $C$ modulo $p$ and the splitting behavior modulo $p$ of $f$ and $\psi_f$; we also show that for primes $\equiv 1 \pmod{3}$ the matrix $A_p$ suffices and that for primes $\equiv 2 \pmod{3}$ the genericity assumption on $C$ is unnecessary. An element of the proof, which may be of independent interest, is the determination of the density of the set of primes of ordinary reduction for a generic Picard curve. By combining this with recent work of Sutherland, we obtain a practical deterministic algorithm that computes $Z(C_p,T)$ for almost all primes $p \le N$ using $N\log(N)^{3+o(1)}$ bit operations. This is the first practical result of this type for curves of genus greater than 2.
\end{abstract}
\maketitle

\section{Introduction}

Let $C$ be a Picard curve defined over $\Q$, that is, a curve admitting an affine model given by an equation of the form
\begin{equation}\label{equation: Picardcurve}
y^3=f(x)\,,
\end{equation}
where $f\in \Q[x]$ is a degree $4$ separable polynomial. Without loss of generality we will assume from now on that $C$ is given by an integral model in which $f(x)=x^ 4+f_2x^ 2+f_1x+f_0$ and the $f_i$ are integers. Let $\zeta_3$ denote a primitive cubic root of unity in an algebraic closure $\Qbar$ of $\Q$. The obvious action of $\langle \zeta_3\rangle$ on~$C$ induces a ring monomorphism from $\Z[\zeta_3]$ to the geometric endomorphism ring $\End(\Jac(C)_\Qbar)$ of the Jacobian of $C$. We will say that~$C$ is \emph{generic} if this ring monomorphism is an isomorphism. This amounts to asking that the geometric endomorphism algebra $\End(\Jac(C)_\Qbar)\otimes \Q$ is isomorphic to $\Q(\zeta_3)$.

For a prime $p$ of good reduction for $C$, let $C_p$ denote the reduction of $C$ modulo~$p$. The zeta function of $C_p$ is a formal power series
$$
Z(C_p,T):=\exp\left( \sum_{n\geq 1}\#C_p(\F_{p^n})\frac{T^n}{n}\right) \in\Q[[T]]\,,
$$ 
defined in terms of the number of points of $C_p$ over the finite extensions of the finite field of $p$ elements $\F_p$. It can be shown to be a rational function admitting an expression of the form
$$
Z(C_p,T)=\frac{L_p(C,T)}{(1-T)(1-pT)}\,,
$$ 
where $L_p(C,T)$ is a degree 6 polynomial with coefficients in $\Z$, which we call the $L$-polynomial of~$C$ at~$p$.

The purpose of the present paper is twofold. On the one hand, we contribute several theoretical results related to the $L$-polynomials of $C$. On the other hand, we apply these results to develop and implement an efficient and practical deterministic algorithm to compute the $L_p(C,T)$, for almost all $p$ up to some bound~$N$. The sequence of the $L_p(C,T)$ contains deep arithmetic information about $C$, and it has been the object of celebrated conjectures in number theory, such as generalized forms of the Sato--Tate conjecture (see for example \cite[Chap. 8]{Ser12}, \cite{KS09}) or of the Lang--Trotter conjecture. Convincing numerical evidence for these conjectures sometimes requires computing $L_p(C,T)$ for $p$ up to a bound $N$ within the range $[2^{30},2^{40}]$. In recent years, there have been several breakthroughs to make these computations feasible for curves of genus at most 2; the methods of previous articles, however, do not allow for computations for $N$ within this range for curves of genus~3. 

Before we return to this question in more detail let us start by describing the theoretical contributions on which the main algorithm of this paper relies.

\subsection*{Theoretical contributions}  Let $\cC_p$ denote the Cartier operator acting on the 3-dimensional $\F_p$-vector space $H^0(C_p,\Omega^1_{C_p/\F_p})$ of regular differentials of $C_p$. By the Cartier--Manin matrix~$A_p$ of $C$ at $p$, we will mean the matrix of the operator~$\cC_p$ acting on this space in a certain basis. 
By the work of Katz and Serre, the reduction of $L_p(C,T)$ \emph{modulo $p$} is uniquely determined by $A_p$ (see Section \ref{section: first facts} for a quick recollection of these facts and their references). 

In the first part of this article (corresponding to Sections \ref{section: preliminaries} and \ref{section: theoretical}), we show that in fact~$A_p$ carries enough information to uniquely determine $L_p(C,T)$ quite often.
In order to state more precisely our main results, let us consider separately the cases $p\equiv 1 \pmod 3$ and $p\equiv 2 \pmod 3$. To this aim, let $\fS(C)$ (resp. $\fI(C)$) denote the set of odd primes coprime to the discriminant of $f$ and congruent to $1$ (resp.~$2$) modulo $3$. In the first case, we obtain the following.
 
\begin{theorem}\label{theorem: main1}
Let $C$ be a Picard curve defined over $\Q$. Then:
\begin{enumerate}[i)]
\item For every $p\geq 53$ in $\fS(C)$ of ordinary reduction for $C$, the Cartier--Manin matrix of $C$ at $p$ uniquely determines the $L$-polynomial $L_p(C,T)$. 
\item If $C$ is generic, then every prime in $\fS(C)$ outside a density 0 set is ordinary for $C$.
\end{enumerate}
\end{theorem}

See Corollary~\ref{corollary: CartierManin} and Corollary \ref{corollary: genericordinary}. Let $g_p(C,T)$ denote the reversed $T^3 \chi_p(1/T)$ of the characteristic polynomial $\chi_p(T)$ of $\cC_p$ acting on $H^0(C_p,\Omega^1_{C_p/\F_p})$. To prove the theorem, we first show that the existence of a functorial map from the crystalline cohomology space $\Hcrys(C_p/\Z_p)$ onto the semisimple subspace $H^0(C_p,\Omega^1_{C_p/\F_p})\sems$ implies that the action of $\Z[\zeta_3]$ on these two spaces induces compatible factorizations of the polynomials $L_p(C,T)$ and $g_p(C,T)$, over $\Z[\zeta_3][T]$ and $\F_p[T]$, respectively. This is the content of Section \ref{section: superelliptic}, which is written for general superelliptic curves of prime exponent (with no cost of extra conceptual or technical complication with respect to the case of Picard curves). 

The first assertion of the theorem is proven in Section \ref{section: splitcase}. It relies on the compatible factorizations of $L_p(C,T)$ and $g_p(C,T)$, and uses the Weil bounds and the fact that if $p$ is ordinary, then $g_p(C,T)$ has degree~$3$. 

While the proof of the first assertion is $p$-adic in nature, the proof of the second statement (accomplished in Section \ref{section: densityquestions}) relies on $\ell$-adic methods. It uses the description of $L_p(C,T)$ in terms of the \'etale cohomology group $\Het(C_\Qbar,\Z_\ell)$ (or alternatively, in terms of the Tate module $T_\ell(\Jac(C))$). It should be regarded as a refinement of \cite[Thm.~1]{Fit20} obtained by replacing Ogus' method by that of Sawin (see \cite{Saw16}). We actually need a mild generalization of Sawin's result, which is presented in Section~\ref{section: Sawin}.

We now turn to primes $p\equiv 2 \pmod 3$. Attached to the Picard curve $C$, define the polynomial
\begin{align*}
\psi_f(x) := x^9 &+ 24f_2x^7 - 168f_1x^6 + (1080f_0 - 78f_2^2)x^5 + 336f_1f_2x^4\\
 &+ (1728f_0f_2 - 636f_1^2 + 80f_2^3)x^3 + (-864f_0f_1 - 168f_1f_2^2)x^2\\\numberthis \label{equation: defpsif}
 &+ (-432f_0^2 + 216f_0f_2^2 - 120f_1^2f_2 - 27f_2^4)x - 8f_1^3\,.
\end{align*}
The splitting field of $\psi_f(x^3/2)$ is the $2$-torsion field of the Jacobian of $C$. This is explained in the Lemma of the Appendix to this article, which was kindly written for us by Andrew Sutherland. We then have the following result.

\begin{theorem}\label{theorem: main2}
Let $C$ be a Picard curve defined over $\Q$, and let $f,\psi_f\in \Z[x]$ be as in \eqref{equation: Picardcurve} and \eqref{equation: defpsif}, respectively. For every prime in $\fI(C)$, the data:
\begin{enumerate}[i)]
\item the Cartier--Manin matrix of $C$ at $p$,
\item the knowledge of $\psi_f$ having a root or not modulo $p$, and 
\item the knowledge of $f$ being irreducible or not modulo $p$
\end{enumerate}
uniquely determine the $L$-polynomial $L_p(C,T)$.
\end{theorem}

Its proof is the content of Section \ref{section: inertcase} and the Appendix. It is based on the following simple idea. For all $p\in \fI(C)$, the $L$-polynomial $L_p(C,T)$ is uniquely determined by the coefficient of $T^2$. By writing this coefficient as $p-t_p$, one has that $|t_p|\leq 2p$. Since condition $i)$ determines $t_p$ modulo $p$, the theorem follows from the facts that~$ii)$ determines $t_p$ modulo $2$ (see the Theorem in the Appendix), and that $iii)$ determines it modulo $3$ (see Lemma~\ref{lemma: fmodp} and Proposition~\ref{proposition: Lpolymod3}). The latter should be no surprise, as it is well known that the splitting field of $f$ is closely related to the $3$-torsion field of the Jacobian of $C$.

We highlight the \emph{constructive} nature of the proofs of Theorems \ref{theorem: main1} and \ref{theorem: main2}. By this, we mean that they provide a way to compute $L_p(C,T)$ from the given data. This is exploited in the second part of the paper.

\subsection*{A practical algorithm} In the second part of the paper (corresponding to Section~\ref{section: algorithmic}), we are concerned with the problem of computing the $L$-polynomials $L_p(C,T)$, for $p\leq N$. 

At a \emph{theoretical} level, this problem is well understood: for a fixed smooth and projective curve of genus $g$ and defined over $\Q$, Pila's algorithm \cite{Pil90} (extending \cite{Sch85}) computes the zeta function at a prime $p$ of good reduction using $\log(p)^{g^{O(1)}}$ operations. 

As mentioned above, there are situations in which one is interested in computing the zeta functions at all primes $p\leq N$. In these situations, one can do better than applying Pila's algorithm prime by prime. In fact, Harvey has proposed an algorithm that achieves this computation using a total of $N\log(N)^{3+o(1)}$ bit operations (see \cite{Har14} for the case of hyperelliptic curves and \cite{Har15} for the case of a general arithmetic scheme, including, of course, the case of smooth projective curves). 

The existence of Pila's and Harvey's theoretical algorithms sets the challenge to develop \emph{practical} versions of them, amenable for implementation and producing effective results when run by real hardware and $N$ is in the range, say, $[2^{30},2^{40}]$. 

Let us summarize part of the progress which has been made toward the obtaining of practical versions of Harvey's algorithm. In the hyperelliptic curve case, practical algorithms to compute the Cartier--Manin matrix $A_p$ have been developed and implemented by Harvey and Sutherland (see \cite{HS14} and \cite{HS16}). In \cite{HMS16}, such practical algorithms were developed for genus $g=3$ geometrically hyperelliptic curves, that is, curves admitting an affine model given by the equations 
$$
h(x,y)=0\,, \qquad w^2=f(x,y)\,, 
$$ 
where $f,h\in \Z[x,y]$ are polynomials of respective degrees 4 and 2. In genus $g\le 2$, computing $A_p$ suffices to compute the $L$-polynomial, but this is no longer true for $g\geq 3$. This makes adapting Harvey's algorithm more difficult.

In another direction, the case of cyclic covers of the projective line has been examined in \cite{ABCMT19}. The authors provide an algorithm to compute the $L$-polynomial at $p$ using $p^{1/2+o(1)}$ bit operations in the case of a superelliptic curve, and this yields the fastest practical algorithm in the literature for computing the $L$-polynomials at all primes $p\leq N$. In \cite{Abe18}, a Las Vegas type algorithm with expected complexity $\log(p)^{14+o(1)}$ is provided to compute the $L$-polynomial at~$p$ of a genus three hyperelliptic curve. However, for computing the $L$-polynomials at all $p\leq N$, when $N$ is in the range we consider, \cite{ABCMT19} still exhibits better performance. 

For a Picard curve $C$ defined over $\Q$, our case of interest, \cite{BTW05} provides an algorithm of complexity $O(p^{1/2})$ to compute $L_p(C,T)$. The main computational contribution of this article is a practical deterministic algorithm for the computation of the $L$-polynomials of a generic Picard curve at \emph{almost} all primes $p\leq N$, performing $N\log(N)^{3+o(1)}$ bit operations. This is obtained by combining the recent work of Sutherland \cite{Sut20} with the constructive proofs of Theorems~\ref{theorem: main1} and~\ref{theorem: main2}.

Before stating the result more precisely, let us establish some conventions and notation. From now on, the term \emph{algorithm}, without further qualification, is used to refer to a deterministic algorithm. For a generic Picard curve $C$ defined over $\Q$, let $\fSst(C)$ denote the subset of ordinary primes of $\fS(C)$. For $N\geq 1$, set
$$
\fSst_N(C):=\fSst(C)\cap [1,N]\,,\qquad \fI_N(C):=\fI(C)\cap [1,N]\,.
$$

\begin{theorem} Let $C$ be a generic Picard curve.
 Algorithm \ref{algorithm: full} determines $\fSst_N(C)$ and returns $L_p(C,T)$ for every prime $p\in \fSst_N(C)\cup \fI_N(C)$ using $N\log(N)^{3+o(1)}$ bit operations. 
\end{theorem}

One might speculate that the time spent with the computation of the $L_p(C,T)$ for primes $p\leq N$ in the complement of the set $\fSst_N(C) \cup \fI_N(C)$ using existing algorithms (such as \cite{ABCMT19} or lifting methods; see Remark \ref{remark: especulation1}) would be subsumed by the bound $N\log(N)^{3+o(1)}$. Despite showing that this complement is of 0 density, we were not able to prove that it is thin enough to retrieve such a conclusion.
Algorithm~\ref{algorithm: full} 
is described in Section \ref{section: algdesc}, where we also analyze its correctness and running time. In Section \ref{section: implementation}, we discuss our implementation of the algorithm and its speed compared to the implementation in \cite{Sut20}. In Section \ref{section: implementation}, we also explain how the methods of this article can be combined with the algorithm of \cite{ABCMT19} to provide a constant factor  improvement in the performance of the latter.

\section{Preliminaries}\label{section: preliminaries}

Fix an algebraic closure $\Qbar$ of $\Q$ and let $\overline \Z$ denote its ring of algebraic integers. For a rational prime $p$, let $\F_p$ denote the finite field with $p$ elements, and let $v_p$ denote a prime ideal of $\overline \Z$ lying above $p$ (by abuse of notation, we will also denote by~$v_p$ the corresponding extension of the $p$-adic valuation). The residue field of $v_p$ is an algebraic closure of $\F_p$, which we will denote by $\overline \F_p$. 

Let $D_{v_p}$ denote the decomposition group of $v_p$. We will denote by $\varphi_p$ the (arithmetic) Frobenius element of $G_{\F_p}:=\Gal(\overline \F_p/\F_p)$ and by $\Frob_p$ a preimage in $D_{v_p}$ of~$\varphi_p$ by the canonical projection from $D_{v_p}$ to $G_{\F_p}$.

\subsection{\texorpdfstring{$L$-polynomials modulo $p$}{L-polynomials modulo p}}\label{section: first facts} Throughout this section, $C$ denotes a smooth and projective curve of genus~$g$ defined over $\Q$. Let $S$ denote a finite set of primes such that $C$ has good reduction outside $S$. By this, we mean that there exists a smooth and projective scheme $\mathcal C \rightarrow \mathrm{Spec}(\Z)- S$ whose generic fiber is $C$. We will denote by $C_{p}$ the special fiber of $\mathcal C$ at~$p$, and will refer to it as the reduction of $C$ modulo~$p$. We will denote by $\Jac(C)$ the Jacobian of $C$. 

\subsection*{Reduction from \'etale cohomology} Let $\ell$ denote a rational prime. Suppose from now on that $p\not =\ell$ is a prime of good reduction for $C$. Let $V_\ell(\Jac(C))$ denote the rational $\ell$-adic Tate module of $\Jac(C)$. We define the $L$-polynomial of $C$ at~$p$ as
\begin{equation}\label{equation: Lpolynomial}
L_p(C,T):=\det(1-\Frob_pT\, |\,V_\ell(\Jac(C)))=\det(1-\Frob_p^{-1}T\, |\,\Het(\Cbar,\Z_\ell))\,.
\end{equation}
Here $\Cbar$ stands for the base change of $C$ from $\Q$ to~$\Qbar$. It is a degree $2g$ polynomial with integer coefficients, and it does not depend on the choice of $\ell$. Let us write 
$$
g_p(C,T):=\det(1-\varphi_p^{-1}T\,|\,\Het(\Cbar_{p},\Z/p\Z))\,,
$$
where $\Cbar_{p}$ is the base change of $C_{p}$ from $\F_p$ to~$\Fbar_p$. 
By \cite[Thm. 3.1]{Kat73}, we have the congruence
\begin{equation}\label{equation: expose22}
L_p(C,T)\equiv g_p(C,T) \quad \pmod p\,.
\end{equation}

We will give an alternative description of $g_p(C,T)$ in terms of the Cartier operator. Let $F_p:C_p\rightarrow C_p$ be the absolute Frobenius, the map which is the identity on the underlying topological space of $C_p$ and which acts by $F_p(\nu)=\nu^p$ on sections $\nu$ of $\cO_{C_p}$. By abuse of notation, let us also denote by $F_p$ the map
$$
 H^1(C_p,\cO_{C_p})\rightarrow H^1(C_p,\cO_{C_p})
$$
induced in cohomology. Let $\cC_p\colon H^0(C_p,\Omega^1_{C_p/\F_p})\rightarrow H^0(C_p,\Omega^1_{C_p/\F_p})$ denote the Cartier operator as defined in \cite[\S10]{Ser58} (see also \cite[\S2]{Sut20} or \cite[\S2]{AH19} for a concise treatment). 
The $\F_p$-linear operators $F_p$ and $\cC_p$ give rise to $\Fbar_p$-linear operators 
$$
\begin{array}{l}
F_p\otimes \Fbar_p\colon H^1(\Cbar_p,\cO_{\Cbar_p})\rightarrow H^1(\Cbar_p,\cO_{\Cbar_p})\,,\\[6pt]
\cC_p\otimes \Fbar_p\colon H^0(\Cbar_p,\Omega^1_{\Cbar_p/\Fbar_p})\rightarrow H^0(\Cbar_p,\Omega^1_{\Cbar_p/\Fbar_p})\,.
\end{array}
$$
We denote with a subscript $\sems$ the semisimple part of the above spaces with respect to the respective operators. Fundamental to our discussion will be the existence of isomorphisms
\begin{equation}\label{equation: comparisonisomorphisms}
\Het(\Cbar_p,\Z/p\Z)\otimes_{\F_p} \Fbar_p \simeq H^1(\Cbar_p,\cO_{\Cbar_p})\sems \simeq H^0(\Cbar_p,\Omega^1_{\Cbar_p/\Fbar_p})\sems \,.
\end{equation}
For the first isomorphism we refer to \cite[Prop. 2.2.5]{Kat73}, where it is moreover shown that it transforms the automorphism $\varphi_p^{-1}\otimes \Fbar_p$ into the automorphism $F_p\otimes \Fbar_p$. The second isomorphism is Serre duality, and by \cite[Prop. 9]{Ser58} it sends $F_p\otimes \Fbar_p$ to $\cC_p\otimes \Fbar_p$. We deduce that
$$
g_p(C,T)=\det(1-\cC_pT\,|\,H^0(C_{p},\Omega^1_{C_p/\F_p})\sems)=\det(1-\cC_pT\,|\,H^0(C_{p},\Omega^1_{C_p/\F_p}))\,.
$$ 

\subsection*{Reduction from crystalline cohomology} Crystalline cohomology provides a functorial version of \eqref{equation: expose22}, which we now recall. Let $W$ denote the ring of Witt vectors of $\F_p$, and let $\Hcrys(C_p/W)$ denote the crystalline cohomology of $C_p$. The latter is a free module over $W\simeq \Z_p$ of rank $2g$. By abuse of notation, we also denote by 
$$
F_p\colon \Hcrys(C_p/W)\rightarrow \Hcrys(C_p/W)
$$
the map induced by the absolute Frobenius. By \cite{KM74}, we have the equality
\begin{equation}\label{equation: cryspoly}
L_p(C,T)=\det(1-F_pT\,|\,\Hcrys(C_p/W))\,.
\end{equation}
The cohomology group $\Hcrys(C_p/W)$ comes equipped with a functorial map to de Rham cohomology
\begin{equation}\label{equation: functorialprojaux}
\Hcrys(C_p/W)\rightarrow \HdR(C_p/\F_p)\rightarrow \HdR(C_p/\F_p)\sems
\end{equation}
preserving the action of the absolute Frobenius on the respective spaces. By \cite[(3.3.4)]{Kat73} and Serre duality, respectively, there are functorial isomorphisms
$$
\HdR(C_p/\F_p)\sems\simeq H^1(C_p,\cO_{C_p})\sems\simeq H^0(C_p,\Omega^1_{C_p/\F_p})\sems
$$
mapping $F_p$ to $\cC_p$. Together with \eqref{equation: functorialprojaux}, this provides a functorial map
\begin{equation}\label{equation: functorialproj}
\Hcrys(C_p/W)\rightarrow H^0(C_p,\Omega^1_{C_p/\F_p})\sems
\end{equation}
mapping the absolute Frobenius $F_p$ to the Cartier operator $\cC_p$, which refines \eqref{equation: expose22} and which will be exploited in Section \ref{section: superelliptic}.

Note that further subtleties arise when working over a nonprime field, but we will not encounter them in our discussion.

\subsection{Superelliptic curves and Cartier--Manin matrices}\label{section: superelliptic}

From now on, let $C$ be a superelliptic curve defined over $\Q$, that is, a curve admitting an affine model given by the equation
$$
y^m=f(x)\,,
$$
where $m\geq 2$ is an integer and $f\in \Z[x]$ is a separable polynomial of degree $d\geq 3$. Let $\cO$ denote $\Z[\zeta_m]$, where $\zeta_m$ denotes a fixed primitive $m$th root of unity in~$\overline \Z$. We will further assume that $m$ is prime. Then (see \cite[p. 149]{Sch98}) the action of $\langle \zeta_m\rangle$ on $C$ induces a ring monomorphism
\begin{equation}\label{equation: ring homomorphism}
\cO\hookrightarrow \End(\Jac(C)_{\Q(\zeta_m)})\,.
\end{equation}

The genus of $C$ is then expressed by the formula
$$
g=\frac{(d-1)(m-1)-\gcd(m,d)+1}{2}\,.
$$

Let $\fS(C)$ denote the set\footnote{Note that the set $\fS(C)$ in fact depends on the model chosen for $C$.} of rational  primes coprime to the discriminant of~$f$ and to its leading coefficient, and congruent to $1$ modulo $m$. Assume until the end of this section that~$p$ belongs to $\fS(C)$. In particular, $p$ is a prime of good reduction for $C$ which splits completely in $\Q(\zeta_m)$. The choice of a prime~$v_p$ of $\Zbar$ lying over~$p$ singles out one of the $\varphi(m)$ primes of $\cO$ lying over~$p$, where $\varphi$ denotes Euler's totient function. Call it $\p_0$. For a prime~$\p$ of~$\cO$ lying over $p$, define the $\F_p$-algebra map 
\begin{equation}\label{equation: sigmamap}
\sigma_\p\colon \cO\otimes_\Z\F_p\simeq \bigoplus_{\p'\mid \p}\F_{\p'}\xrightarrow{\pr_{\p}} \F_p\,,
\end{equation}
where $\pr_{\p}$ denotes the projection from the $\F_\p$-component. To shorten the notation, we will often simply write $\sigma$ to denote $\sigma_\p$, and $\sigma_0$ to denote $\sigma_{\p_0}$. Note that there exists an integer $1\leq j(\sigma)\leq m$ such that the equality
\begin{equation}\label{equation: jsigma}
\sigma(\zeta_m)=\sigma_{0}(\zeta_m)^{j(\sigma)}
\end{equation}
holds in $\F_p$. Fix a $W$-algebra map $\cO\otimes_\Z W\rightarrow W$ rendering the diagram
\begin{equation}\label{equation: sigmalift}
\xymatrix{
\cO\otimes_{\Z} W\ar[d]^{\otimes_W \F_p}\ar[r] & W \ar[d]^{\otimes_W \F_p}\\
\cO \otimes_{\Z}\F_p\ar[r]^{\sigma} & \F_p
}
\end{equation}
commutative. Let us still denote by $\sigma$ the map $\cO\otimes_\Z W\rightarrow W$.

We will use $\sigma_0$ to define an action of the cyclic group $\langle \zeta_m\rangle$ on $C_p$. It is given by
$$
[\zeta_m]_{\sigma_0}\cdot (x,y)= (x,\sigma_0(\zeta_m)y)\,.
$$ 
Define the $\sigma$-eigenspaces
$$
\begin{array}{l}
H^1_\sigma(C_p/W):=\Hcrys(C_p/W)\otimes_{\cO\otimes_\Z W,\sigma}W\,,\\[6pt]
H^0_{\sigma}(C_p,\Omega^1_{C_p/\F_p}):=H^0(C_p,\Omega^1_{C_p/\F_p})\otimes_{\cO\otimes_\Z \F_p,\sigma} \F_p\,.
\end{array}
$$
The commutativity of \eqref{equation: sigmalift} yields a commutative diagram
$$
\xymatrix{
\Hcrys(C_p/W)\ar[d]\ar[rr]^{\otimes_W \F_p} && \Hcrys(C_p/W)\otimes_W\F_p  \ar[d]\\
H_\sigma^1(C_p/W) \ar[r]^-{\otimes_W \F_p} & H_\sigma^1(C_p/W)\otimes_W\F_p \ar[r]^-{\simeq} & (\Hcrys(C_p/W)\otimes_W\F_p)\otimes_{\cO\otimes \F_p,\sigma }\F_p\,,
}
$$
where the vertical arrows are projections to the repective $\sigma$-eigenspaces.
By combining the bottom row of the above diagram with the functorial map \eqref{equation: functorialproj}, we obtain a map
\begin{equation}\label{equation: theseekedone}
H^1_\sigma(C_p/W)\rightarrow H^0_{\sigma}(C_p,\Omega^1_{C_p/\F_p})\sems
\end{equation}
that sends the absolute Frobenius $F_p$ to the Cartier operator $\cC_p$. Set the polynomials
\begin{equation}\label{equation: polysLandg}
\begin{array}{l}
L_p^\sigma(C,T):=\det(1-F_pT\,|\,H^1_\sigma(C_p/W))\,,\\[4pt]
g_p^\sigma(C,T):=\det(1-\cC_p T\,|\,H^0_\sigma(C_p,\Omega^1_{C_p/\F_p})\sems)=\det(1-\cC_p T\,|\,H^0_\sigma(C_p,\Omega^1_{C_p/\F_p}))\,.
\end{array}
\end{equation}
We now describe the basis of $H^0(C_p,\Omega^1_{C_p/\F_p})$ given by \cite[Lem. 6]{Sut20}. Namely, set 
$$
\mu:=m-\left\lfloor \frac{m}{d}\right\rfloor -1\,,\quad\text{and}\quad d_j:=d-\left\lfloor \frac{dj}{m}\right\rfloor-1\text{ for $1\leq j\leq \mu$,}
$$ 
and then define
\begin{equation}\label{equation: basis}
\omega_{i,j}=x^{i-1}y^{j-m}dx\qquad \text{for }1\leq j\leq \mu,\, 1\leq i\leq d_j\,.
\end{equation} 
We then denote by $\omegab$ the basis $(\omega_{11}, \omega_{12},\dots,\omega_{21},\dots)$, where the $\omega_{ij}$ are lexicographically ordered by subindex.
Let $A_p$ be the matrix of the Cartier operator $\cC_p$ acting on $H^0(C_p,\Omega^1_{C_p/\F_p})$ with respect to $\omegab$. We refer to $A_p$ as the \emph{Cartier--Manin matrix} of $C_p$ with respect to $\omegab$. 

Let $\omegab^{\sigma}$ denote the tuple $(\omega_{i,j(\sigma)})_i$, where $i$ runs over the interval $[1,d_{j(\sigma)}]$ and $j(\sigma)$ is as defined in \eqref{equation: jsigma}. From the relation
$$
[\zeta_m]^*_{\sigma_0}(\omega_{i,j(\sigma)})=\sigma_0(\zeta_m)^{j(\sigma)}\omega_{i,j(\sigma)}=\sigma(\zeta_m)\omega_{i,j(\sigma)}\,,
$$
we obtain that $\omegab^\sigma$ is a basis of the $\sigma$-eigenspace $H^0_\sigma(C_p,\Omega^1_{C_p/\F_p})$. Let us denote by $A_p^{\sigma}$ the matrix of the restriction of $\cC_p$ on this subspace with respect to~$\omegab^{\sigma}$.

\begin{lemma}\label{lemma: decred}
The following holds:
\begin{enumerate}[i)]
\item $L_p^\sigma(C,T)$ is a polynomial with coefficients in $\cO$ and degree $2g/\varphi(m)$.\vspace{0,1cm}
\item We have factorizations
$$
L_p(C,T)=\prod_{\sigma}L_p^\sigma(C,T)\,,\qquad g_p(C,T)=\prod_{\sigma}g_p^\sigma(C,T)
$$
where $\sigma=\sigma_\p:\cO\otimes_\Z W\rightarrow W$ runs through the $W$-algebra maps defined in~\eqref{equation: sigmalift}. Moreover, for every such $\sigma$ we have
$$
L_p^{\sigma}(C,T)\equiv g_p^\sigma(C,T)\equiv T^{d_{j(\sigma)}}\cdot \chi_p^\sigma(1/T)\qquad \pmod{\p_0}\,,
$$
where $\chi_p^\sigma$ is the characteristic polynomial of $A_p^\sigma$.
\end{enumerate}
\end{lemma}

\begin{proof}
Given the ring homomorphism \eqref{equation: ring homomorphism}, in the ``$\ell$-adic setting", part $i)$ is a special case of \cite[Thm. 2.11, Thm. 2.12]{Rib76}. The same arguments apply in the present ``$p$-adic setting". The first part of $ii)$ follows from the eigenspace decompositions
$$
H^1(C_p/W)\simeq \bigoplus_{\sigma} H^1_\sigma(C_p/W)\,,\qquad H^0_\sigma(C_p,\Omega^1_{C_p/\F_p})\simeq \bigoplus_{\sigma}H^0_\sigma(C_p,\Omega^1_{C_p/\F_p})\,.
$$ 
As for the second part, by \eqref{equation: theseekedone}, we have a congruence 
$$
L_p^\sigma(C,T)\equiv g^{\sigma}_p(C,T)\qquad \pmod {\p_0}\,,
$$
and the discussion in the paragraph preceeding this lemma implies that $g_p^\sigma(C,T)$ is the reversed characteristic polynomial of $A_p^\sigma$.
\end{proof}

\section{\texorpdfstring{$L$}{L}-polynomials of Picard curves: some theoretical results}\label{section: theoretical}

From now on, let $C$ be a Picard curve defined over $\Q$, that is, a curve admitting an affine model given by the equation
$$
y^3=f(x)\,,
$$ 
where $f(x)=x^ 4+f_2x^2+f_1x+f_0$ is a separable polynomial in which the $f_i$ are integers. Thus $C$ is a superelliptic curve for which $m=3$ and $d=4$, and accordingly we will denote $\Z[\zeta_3]$ by $\cO$ from now on.

\subsection{The split case}\label{section: splitcase} Recall from the introduction that $\fS(C)$ denotes the set of primes coprime to the discriminant of $f$ and congruent to $1$ modulo $3$. Resume also the notations from Section \ref{section: superelliptic}, with the following slight modifications: for a prime~$p$ in $\fS(C)$ we simply denote by $\p$ the prime of $\cO$ singled out by $v_p$ and by $\pbar$ the prime of $\cO$ such that $p\cO=\p\pbar$; we write~$\sigma$ and $\sigmabar$ for the maps $\sigma_\p$ and $\sigma_{\pbar}$. We also fix once and for all generators $\pi,\overline \pi$ in $\cO$ of the ideals~$\p,\pbar$.

With the notation of Section \ref{section: superelliptic}, we have that $d_1=2$ and $d_2=1$. Set $n_1:=(2p-2)/3$ and $n_2:=(p-1)/3$. Let $\tilde f\in \F_p[x]$ denote the reduction of $f\in \Z[x]$ modulo $p$ and let $f_a^n$ denote the coefficient of $x^a$ in $\tilde f(x)^n$. From \cite[(8)]{Sut20} we see that
\begin{equation}\label{equation: CartierManinSplit}
A_p^\sigma=\begin{pmatrix}
f_{p-1}^{n_1} & f_{p-2}^{n_1}\\
f_{2p-1}^{n_1} & f_{2p-2}^{n_1}
\end{pmatrix}\,,
\qquad
A_p^\sigmabar=\begin{pmatrix}
f_{p-1}^{n_2}
\end{pmatrix}\,.
\end{equation}
Recall the polynomials defined in \eqref{equation: polysLandg}. Note that $L^\sigma_p(C,T)$ and $L^\sigmabar_p(C,T)$ are complex conjugate to each other. Let us define their coefficients as:
$$
\begin{array}{l}
L^\sigma_p(C,T)=:1-a_\p T+b_\p T^2-c_\p T^3\,,\\[4pt]
L^\sigmabar_p(C,T)=:1-a_\pbar T+b_\pbar T^2-c_\pbar T^3\,,\\[4pt]
g^\sigma_p(C,T)=:1-r_\p T+s_\p T^2\,,\\[4pt]
g^\sigmabar_p(C,T)=:1-r_\pbar T\,.\\[4pt]
\end{array}
$$

Recall that a prime $p$ is said to be ordinary for $C$ if the central coefficient of $L_p(C,T)$ is not divisible by $p$, or equivalently if the polynomial $g_p(C,T)\in \F_p[T]$ has degree $3$. Our goal is to show that, for an ordinary prime $p$, the $L$-polynomial $L_p(C,T)$ can be recovered from the Cartier--Manin matrix.

\begin{lemma}\label{lemma: constantterm}
Let $p$ in $\fS(C)$ be of ordinary reduction for $C$. Then there exists a sixth root of unity $\zeta$ such that
\begin{equation}\label{equation: Heckechar}
c_\p=\zeta \overline\pi p\,.
\end{equation}
\end{lemma}

\begin{proof}
Since $c_\p$ divides $p^3$ and it is the product of three $p$-Weil numbers, there exists an integer $0\leq i \leq 3$ such that $c_\p=\zeta\pi^i\overline\pi^{3-i}$, where $\zeta$ is a unit of $\cO$ and thus a sixth root of unity. Since $g^\sigma_p(C,T)$ has degree $< 3$, we know that $i\geq 1$. Proving the lemma amounts to showing that $i=1$. Let $\alpha_\p$ denote a reciprocal root of $L_p^\sigma(C,T)$. Since $\alpha_p$ is an algebraic integer and $\alpha_\p\overline \alpha_\p=p$, we have $0\leq v_p(\alpha_p)\leq 1$. Since $p$ is ordinary, the polynomial $g^\sigma_p(C,T)$ has degree 2. The usual Newton polygon argument then shows that two of the reciprocal roots of $L_p^\sigma(C,T)$ have $v_p$-adic valuation $0$. Hence the $v_p$-adic valuation of $c_\p$ is that of the third reciprocal root of $L_p^\sigma(C,T)$, which as argued before is at most $1$. The Lemma follows. 
\end{proof}

\begin{lemma}\label{lemma: bpord}
For $p$ in $\fS(C)$ the following hold:
\begin{enumerate}[i)]
\item $pb_\p=c_\p \cdot a_\pbar$.
\item If $b_\p$ is not divisible by $\pi$, then $a_\pbar$ is not divisible by $\pi$.
\item $p$ is ordinary for $C$ if and only if $b_\p$ is not divisible by $\pi$.
\item $p$ is ordinary for $C$ if and only if $b_\p$ (or, equivalently, $b_\pbar$) is not divisible by $p$. 
\end{enumerate}
\end{lemma}

\begin{proof}
Complex conjugation interchanges the roots of $L_p^\sigma(C,T)$ and those of $L_p^\sigmabar(C,T)$. Therefore we can pair each reciprocal root $\alpha_\p$ of $L_p^\sigma(C,T)$ with a reciprocal root $\alpha_\pbar$ of $L_p^\sigmabar(C,T)$ in such a way that $\alpha_\p\cdot \alpha_{\pbar}=p$, and this implies $i)$ (c.f. \cite[(3.3)]{BTW05}). The coprimality of $\pi$ and $\overline\pi$ and part~$i)$, imply that $v_p(a_\pbar)=1+v_p(b_\p)-v_p(c_\p)$. Since $g_p^\sigma(C,T)$ has degree $\leq 2$, we have $v_p(c_\p)\geq 1$. Thus $v_p(a_\pbar)\leq v_p(b_\p)$, and $ii)$ follows. By Lemma \ref{lemma: decred}, the prime $p$ is ordinary for~$C$ if and only if
$$
g_p^\sigma(C,T)\cdot g_p^\sigmabar(C,T)=g_p(C,T)\in \F_p[T]
$$
has degree 3. This happens if and only if $s_\p,r_\pbar$ are nonzero, which by $ii)$ amounts to saying that $b_\p$ is not divisible by $\pi$. This shows $iii)$. Finally, note that since $b_\p$ is always divisible by $\overline \pi$, divisibility by $p$ amounts to divisibility by $\pi$.
\end{proof}

\begin{corollary}\label{corollary: CartierManin}
Let $p\geq 53$ be a prime in $\fS(C)$ ordinary for $C$. Then the Cartier--Manin matrix $A_p$ of $C$ at $p$ uniquely determines the $L$-polynomial $L_p(C,T)$.
\end{corollary}

\begin{proof}
We will in fact show that $g^\sigma_p(C,T)$ and $g^\sigmabar_p(C,T)$ already uniquely determine $L_p(C,T)$. By Lemma \ref{lemma: decred} and part $i)$ of Lemma \ref{lemma: bpord}, in order to determine $L_p(C,T)$ it suffices to determine $a_\p$ and $c_\p$. Write $a_\p=x+\zeta_3 y$ for some integers~$x$ and~$y$. Recall the map $\sigma\colon \cO\otimes \F_p\rightarrow \F_p$, originating from the prime of $\cO$ singled out by $v_p$. Recall that the generator of this prime has been denoted $\pi$. Given $\alpha\in \cO$, by abuse of notation, let us write $\sigma(\alpha)$ to denote $\sigma(\alpha\otimes 1)$. Then the reductions of $x$ and $y$ modulo $p$ are uniquely determined by the invertible linear system
\begin{equation}\label{equation: linear system}
\begin{pmatrix}
1 & \sigma(\zeta_3)\\
1 & \sigmabar(\zeta_3)
\end{pmatrix}
\begin{pmatrix}
x\\
y
\end{pmatrix}
=
\begin{pmatrix}
r_\p\\
r_\pbar
\end{pmatrix}
\,.
\end{equation}
From the trivial inequalities $(2y-x)^2\geq 0$ and $(2x-y)^2\geq 0$ and the Weil bound $|a_\p|\leq 3\sqrt p$, one finds that
$$
\max\{x^2,y^2\}\leq \frac{4}{3}(x^2+y^2-xy)=\frac{4}{3}|a_\p|^2\leq 12p\,.
$$ 
Therefore the reductions of $x$ and $y$ modulo $p$ determine $x$ and $y$ uniquely as soon as $p\geq 53$. The coefficient $c_\p$ is uniquely determined by the sixth root of unity $\zeta\in \cO$ appearing in \eqref{equation: Heckechar}. Note that $\zeta$ is determined by its image $\sigma(\zeta)\in \F_p$. Using that $p$ is ordinary, by Lemma \ref{lemma: bpord} the equation
\begin{equation}\label{equation: determinezeta}
\sigma(\zeta)= \frac{\sigma(b_\p)}{\sigma(a_\pbar)\sigma(\overline \pi)}=\frac{s_\p}{r_\pbar\sigma(\overline \pi)}
\end{equation}
makes sense in $\F_p$, and it determines $\sigma(\zeta)$. 
\end{proof}

In Section \ref{section: algorithmic}, Corollary \ref{corollary: CartierManin} will be used to describe an algorithm to compute the $L_p(C,T)$ for ordinary primes $p$ in $\fS(C)$. 

In Section \ref{section: densityquestions} we will study the density of ordinary primes. In order to do so, we will need to apply an analogue of Lemma~\ref{lemma: bpord} to certain ``$\lambda$-adic counterparts" of the polynomials $L^\sigma_p(C,T)$. We conclude this section by defining these polynomials and discussing analogues of Lemma~\ref{lemma: constantterm} and Lemma~\ref{lemma: bpord} in the $\lambda$-adic setting. Let $\ell\not =p$ be a prime totally split in $\cO$, and let $\lambda$ and $\overline\lambda$ denote the primes of $\cO$ lying above~$\ell$. Denote by $\cO_\lambda$ the completion of $\cO$ at $\lambda$ and consider the module
\begin{equation}\label{equation: spacefirst}
H^1_\lambda(\Cbar):=\Het(\Cbar,\Z_\ell)\otimes_{\cO\otimes \Z_\ell,\sigma_\lambda}\Z_\ell\,.
\end{equation}
The above tensor product is taken with respect to the $\Z_\ell$-algebra map
$$
\sigma_\lambda\colon \cO\otimes \Z_\ell \simeq \cO_\lambda \oplus \cO_{\overline\lambda}\xrightarrow{\pr_\lambda} \Z_\ell\,,
$$
where $\pr_\lambda$ denotes projection from the $\cO_\lambda$-component.
We define an action of the absolute Galois group $G_{\Q(\zeta_3)}$ on $H^1_\lambda(\overline C)$ by letting it act naturally on $\Het(\Cbar,\Z_\ell)$ and trivially on $\Z_\ell$. Since $p$ splits in $\cO$, we have $\Frob_p\in D_{v_p}\subseteq G_{\Q(\zeta_3)}$, and the polynomial
$$
L_p^\lambda(C,T):=\det(1-\Frob^{-1}_pT\,|\,H^1_{\lambda}(\Cbar))
$$
is well defined. By \cite[Thm. 2.11, Thm. 2.12]{Rib76}, it has coefficients in $\cO$ and degree $3$.
Define similarly $H^1_{\lambdabar}(\Cbar)$ and $L^\lambdabar_p(C,T)$. The decomposition $\Het(\Cbar,\Z_\ell)\simeq H^1_\lambda(\Cbar)\oplus H^1_\lambdabar(\Cbar)$ implies that
$$
L_p(C,T)=L_p^\lambda(C,T)\cdot L_p^{\lambdabar}(C,T)\,.
$$
Note that the polynomials $L^\lambda_p(C,T)$ and $ L^\lambdabar_p(C,T)$ are complex conjugate to each other. Set the following notation for their coefficients
$$
L^\lambda_p(C,T)=:1-a_{\p,\lambda} T+b_{\p,\lambda} T^2-c_{\p,\lambda} T^3\,,\quad L^\lambdabar_p(C,T)=:1-a_{\p,\lambdabar} T+b_{\p,\lambdabar} T^2-c_{\p,\lambdabar} T^3\,.
$$ 

\begin{remark}\label{remark: lambdaconstant} 
Let $\delta_\lambda$ denote $\det(H^1_\lambda(\Cbar))$. By \cite[Thm. 14]{Fit20}, the $1$-dimensional representation~$\delta_\lambda$ is a Hecke character of infinity type equivalent to $(1,2)$. After interchanging $\lambda$ and $\lambdabar$ if necessary, this implies that 
$$
\delta_\lambda(\Frob_p)=\zeta \overline\pi p\,,
$$
where $\zeta$ is a root of unity in $\cO$, and hence a sixth root of unity. It follows that $c_{\p,\lambda}=\delta_\lambda(\Frob_p)=\zeta \overline\pi p$, in analogy with Lemma \ref{lemma: constantterm}.  
\end{remark}

\begin{remark} By unique factorization in $\cO[T]$, if $L_p^\sigma(C,T)$ is irreducible, then after interchanging $\lambda$ and $\lambdabar$ we have that 
\begin{equation}\label{equation: comparisonpolys}
L_p^{\sigma}(C,T)=L_p^{\lambda}(C,T)\,.
\end{equation}
Hence when $L_p^\sigma(C,T)$ is irreducible, \eqref{equation: comparisonpolys} together with Remark \ref{remark: lambdaconstant} provides an alternative proof of Lemma~\ref{lemma: constantterm}. It may be interesting to investigate to which further generality equality \eqref{equation: comparisonpolys} holds, which may be possible via comparison isomorphisms in cohomology. Instead of taking this approach in this paper, we will show that the polynomials $L_p^\lambda(C,T)$ and $L_p^\lambdabar(C,T)$ satisfy the same desired properties that $L_p^\sigma(C,T)$ and $L_p^\sigmabar(C,T)$ satisfy. This is explained in the next remark.
\end{remark}

\begin{remark}\label{remark: analogue}
A glimpse at the proof of Lemma~\ref{lemma: bpord} shows that the only properties of $L_p^{\sigma}(C,T)$ and $L_p^{\sigmabar}(C,T)$ that it uses are that they belong to $\cO[T]$, that they are complex conjugates to each other, that their product is $L_p(C,T)$, and that the reductions of $L^\sigma_p(C,T)$ and $L^\sigmabar_p(C,T)$ modulo $p$ have degree $\leq 2$. That $L_p^{\lambda}(C,T)$ and $L_p^{\lambdabar}(C,T)$ satisfy the first three properties follows from their definition. In virtue of Remark \ref{remark: lambdaconstant}, the fourth property also holds. Therefore Lemma \ref{lemma: bpord} also holds after replacing $a_{\pbar},b_{\p},c_{\p}$ with $a_{\p,\lambdabar},b_{\p,\lambda},c_{\p,\lambda}$. 
\end{remark}

\begin{remark}
It would be interesting to investigate to which extent the results of this section admit generalizations to superelliptic curves of the form $y^4=f(x)$, where $f$ is a degree 3 polynomial. 
\end{remark}

\subsection{The inert case}\label{section: inertcase} Let $\fI(C)$ denote the set of rational primes coprime to the discriminant of $f$ and congruent to $2$ modulo $3$. In particular, any prime of $\fI(C)$ is of good reduction for $C$.

\begin{lemma}\label{lemma: pmod3eq2} 
For every $p$ in $\fI(C)$ there exist algebraic integers $\alpha,\alphabar$ satisfying $\alpha\cdot \overline\alpha=p$ such that
$$
L_p(C,T)=(1+pT^2)(1-\alpha^2T^2)(1-\alphabar^2T^2)\,.
$$
In particular, $L_p(C,T)=(1+pT^2)(1-t_pT^2+p^2T^4)$, where $t_p$ is an integer such that $|t_p|\leq 2p$.
\end{lemma}

\begin{proof}
The statement follows from the fact that $\#C_p(\F_p)=1+p$ and $\#C_p(\F_{p^3})=1+p^3$ (see \cite[Lemma 2.1]{BTW05}).
\end{proof}

Let $\tilde f\in \F_p[x]$ denote the reduction of $f$ modulo $p$.

\begin{lemma}\label{lemma: fmodp}
For every $p$ in $\fI(C)$, we have:
$$
t_p\equiv
\begin{cases}
2 \pmod 3 & \text{ if $\tilde f\in \F_p[x]$ is reducible,}\\
1 \pmod 3 & \text{ if $\tilde f\in \F_p[x]$ is irreducible.}
\end{cases}
$$
\end{lemma}

\begin{proof}
On the one hand, from Lemma \ref{lemma: pmod3eq2} one easily finds that
$$
\#C_p(\F_{p^2})=(1+p)^2-2(\alpha^2+\alphabar^2)\equiv t_p\pmod 3\,. 
$$
On the other hand, since every nonzero element of $\F_{p^2}$ has three distinct cubic roots and $\Cbar_p$ has a single ($\F_p$-rational) point at infinity $P_\infty$, its number of points over~$\F_{p^2}$ satisfies
$$
\#C_p(\F_{p^2})\equiv 1+ n(\tilde f) \pmod 3\,,
$$
where $n(\tilde f)$ is the number of roots of $\tilde f$ over $\F_{p^2}$. The lemma follows from the fact that $n(\tilde f)=0$, $1$, or $4$ depending on whether $\tilde f\in \F_p[x]$ is irreducible, decomposes and has an irreducible factor of degree $3$, or decomposes in factors of degree at most~$2$, respectively. 
\end{proof}

Lemma \ref{lemma: fmodp} tells that for every $p$ in $\fI(C)$ the reduction of $L_p(C,T)$ modulo~$3$ is uniquely determined by the factorization of $\tilde f\in \F_p[x]$. Let us give an alternative proof of this fact, which will also extend to the case that $p$ belongs to $\fS(C)$. The method is an adaptation of the technique used in \cite{CDF20}, where the hyperelliptic case (of arbitrary genus) was considered. 

We will first need to introduce some notations. Let $\Jac(C)$ denote the Jacobian variety of $C$. Let $\alpha_i\in \Qbar$, for $1\leq i\leq 4$, denote the roots of $f\in \Z[x]$, and denote by $E_i$ the degree~$0$ divisor $(\alpha_i,0)-P_\infty$. By \cite[Prop. 2.1]{Aru20}, there is a degree~$3$ effective divisor~$\tilde D_i$ such that 
$$
(1-\zeta_3)\tilde D_i \sim E_i\,.
$$ 
Let $D_i$ denote the degree~$0$ divisor $\tilde D_i-3P_\infty$, and thus also $(1-\zeta_3) D_i \sim E_i$. 

Note that $3E_i, 3D_i\sim 0$, and let $V_1$ denote the subspace of $\Jac(C)[3](\Qbar)$ generated by the classes $[E_i]$. Note that $V_1$ is stable under the Galois action. Let $[\overline D_i]$ denote the image of $[D_i]$ in the quotient $V_2:=\Jac(C)[3](\Qbar)/ V_1$. 

\begin{lemma}\label{lemma: 3torsion}
 $[\overline D_1], [\overline D_2],[\overline D_3]$ is a basis for $V_2$.
\end{lemma}

\begin{proof}
Note that $V_1$ is generated by $[E_1], [E_2], [E_3]$ since $-E_1-E_2-E_3\sim E_4$. In fact, by \cite[Prop. 3.2]{Sch98}, there is a $G_\Q$-equivariant isomorphism
\begin{equation}\label{equation: Schaefer}
\Jac(C)[1-\zeta_3](\Qbar)\simeq V_1
\end{equation}
of $3$-dimensional $\F_3$-vector spaces. Hence it suffices to show that $[\overline D_1], [\overline D_2],[\overline D_3]$ are linearly independent. Suppose there were a nontrivial linear relation among $[\overline D_1], [\overline D_2],[\overline D_3]$. This would yield, by multiplication by $(1-\zeta_3)$, a nontrivial linear relation among $[E_1],[E_2],[E_3]$, but this is absurd.  
\end{proof}

\begin{proposition}\label{proposition: Lpolymod3}
Let $\{d_i\}_i$ denote the set of degrees of the irreducible factors of~$\tilde f$. Then:
\begin{enumerate}[i)]
\item For every $p$ in $\fS(C)$, we have
\[L_p(C,T) \equiv (1-T)^{-2}\prod_i (1-T^{d_i})^2 \pmod{3}\,.\]
\item For every $p$ in $\fI(C)$, we have
\[L_p(C,T) \equiv (1-T^2)^{-1} \prod_{i} (1-T^{d_i})(1-(2T)^{d_i}) \pmod{3}\,.\]
\end{enumerate}
\end{proposition}

\begin{proof}
For any prime $p$ of good reduction, by taking $\ell=3$ in \eqref{equation: Lpolynomial}, we obtain
$$
L_p(C,T) \equiv \det(1-\Frob_p T \mid \Jac(C)[3](\Qbar)) \pmod{3}\,.
$$
By Lemma \ref{lemma: 3torsion}, we need only determine how $\Frob_p$ acts on the basis $[E_i]$ for $V_1$ and~$[\overline D_i]$ for $V_2$. Associated to $\Frob_p$, there is an element $\tau$ in the symmetric group on 4 letters such that $\Frob_p(\alpha_i)=\alpha_{\tau(i)}$. Then $\Frob_p([E_i])=[E_{\tau(i)}]$.
Note that the set of lengths of the cycles in the cycle decomposition of $\tau$ coincides with the set of degrees $\{d_i\}_i$.

Suppose that $p$ is in $\fS(C)$. In this case $\Frob_p(1-\zeta_3)=(1-\zeta_3)$, and hence, from the definition of the $[E_i]$ and $[ D_i]$, we find
$$
\Frob_p([E_i])=[E_{\tau(i)}]\,,\qquad (1-\zeta_3)\Frob_p([ D_i])=(1-\zeta_3)[ D_{\tau(i)}]\,.
$$
By \eqref{equation: Schaefer}, the second equality means that $\Frob_p([ D_i])$ and $[ D_{\tau(i)}]$ difer by an element of $V_1$. Hence $\Frob_p([ \overline D_i])=[ \overline D_{\tau(i)}]$. By Lemma \ref{lemma: auxcharpolys} below applied to $V_1$ and $V_2$, we deduce that 
$$
L_p(C,T) \equiv (1-T)^{-2}\prod_i (1-T^{d_i})^2 \pmod{3}\,.
$$
Suppose that $p$ is in $\fI(C)$. In this case, we have that $\Frob_p(1-\zeta_3) = 1-\zeta_3^2$, and thus we get
$$
(1-\zeta_3^2)\Frob_p([D_i])=\Frob_p((1-\zeta_3)[D_i])=\Frob_p([E_i])= [E_{\tau(i)}]=(1-\zeta_3)[D_{\tau(i)}].
$$
By \eqref{equation: Schaefer}, the above equality means that $(1+\zeta_3)\Frob_p([D_i])$ and $[D_{\tau(i)}]$ differ by an element of $V_1$. Since $(1-\zeta_3)\Frob_p([D_i])$ is an element of $V_1$, we get that $2\Frob_p([\overline D_i])=[\overline D_{\tau(i)}]$, or equivalently that $\Frob_p([\overline D_i])=2[\overline D_{\tau(i)}]$. By Lemma~\ref{lemma: auxcharpolys} below applied to $V_1$ and $V_2$, we deduce that
$$
L_p(C,T) \equiv (1-T)^{-1}(1-2T)^{-1} \prod_{i} (1-T^{d_i})(1-(2T)^{d_i}) \pmod{3}\,,
$$
which completes the proof of the proposition.
\end{proof}

The following lemma is used in the proof of the above proposition.

\begin{lemma}\label{lemma: auxcharpolys}
Let $k$ be a field, $V$ a $k$-vector space of dimension $n$, and $v_1,\dots,v_n$ a basis of $V$. Define $v_{n+1}:=-v_1+\dots - v_n$. Let $f\colon V\rightarrow V$ be the $k$-linear automorphism defined by
$$
f(v_i)=av_{\tau(i)}\,,\qquad \text{ for $i=1,\dots,n$,}
$$
where $a\in k$ and $\tau$ is an element of the symmetric group in $n+1$ letters. Let $\{d_i\}_i$ be the set of lengths of the cycles in the cycle decomposition of $\tau$. Then the reversed characteristic polynomial of $f$ is
$$
\chi_{f,V}(T)=\frac{1}{1-aT}\prod_{i}(1-(aT)^{d_i})\,.
$$
\end{lemma}

\begin{proof}
Let $W$ be a $k$-vector space of dimension $n+1$ with basis $w_1,\dots,w_{n+1}$. Define the $k$-linear maps 
$$
\Phi\colon W\rightarrow V\,,\qquad g\colon W\rightarrow W\,,
$$ 
by $\Phi(w_i)=v_i$ and $g(w_i)= aw_ {\tau(i)}$. It is clear that $\Phi$ is equivariant with respect to the actions of $g$ and $f$ on $W$ and $V$, respectively, and that the kernel $N$ of $\Phi$ is generated by $\sum_i w_i$.
Therefore $\chi_{f,V}(T)=\chi_{g,W}(T)/(1-aT)$, and the lemma follows.
\end{proof}

Recall the polynomial $\psi_f$ defined in \eqref{equation: defpsif}. Let $\tilde \psi_f$ denote its reduction modulo~$p$.

\begin{theorem}\label{theorem: inertLpolys}
Let $C$ be a Picard curve defined over $\Q$, and let $f,\psi_f\in \Z[x]$ be as in \eqref{equation: Picardcurve} and \eqref{equation: defpsif}, respectively. For every prime in $\fI(C)$, the data:
\begin{enumerate}[i)]
\item the Cartier--Manin matrix of $C$ at $p$,
\item the knowledge of $\tilde \psi_f$ having an $\F_p$-rational root or not, and 
\item the knowledge of $\tilde f\in \F_p[x]$ being irreducible or not
\end{enumerate}
uniquely determine the $L$-polynomial $L_p(C,T)$.
\end{theorem}

\begin{proof}
By Lemma \ref{lemma: pmod3eq2}, it suffices to determine $t_p$. Since $|t_p|\leq 2p$, it suffices to detemine $t_p$ modulo $6p$. But $i)$ determines $t_p$ modulo $p$; by the Theorem in the Appendix, $ii)$ determines $t_p$ modulo $2$; and both by Lemma~\ref{lemma: fmodp} or Proposition~\ref{proposition: Lpolymod3}, $iii)$ determines $t_p$ modulo $3$. 
\end{proof}

In Section \ref{section: algorithmic} the above theorem will be used to describe a (deterministic) algorithm to compute the $L_p(C,T)$ for primes $p$ in $\fI(C)$. Let $\Jac(C_p)$ denote the Jacobian variety of $C_p$, and let $n$ (resp. $\lambda$) denote the order (resp. exponent) of the group of $\F_p$-rational points of $\Jac(C_p)$. Below we present a variant of the above theorem in which we replace condition $ii)$ by the knowledge of $\lambda$. 

\begin{theorem}\label{theorem: inertLpolysJacobian}
Let $p\geq 877$ be a prime in $\fI(C)$. The data:
\begin{enumerate}[i)] 
\item $t_p$ modulo $p$,
\item the knowledge of $\tilde f\in \F_p[x]$ being irreducible or not, and
\item the exponent $\lambda$ of $\Jac(C_p)(\F_p)$
\end{enumerate} 
uniquely determine the $L$-polynomial $L_p(C,T)$. 
\end{theorem}

\begin{proof}
Since $n=(1+p)(1+p^2-t_p)$, in order to determine $t_p$, it suffices to determine $n$.
By Lemmas \ref{lemma: pmod3eq2} and \ref{lemma: fmodp}, from $i)$ and $ii)$ we may assume given an integer $0 \leq s_p\leq 3p$ such that $t_p\equiv s_p \pmod {3p}$. In fact,~$s_p$ uniquely determines $t_p$ unless~$p \leq s_p\leq 2p$, and hence we will assume that the latter restriction holds from now on. In this case, the order $n$ is either
\begin{equation}\label{equation: twoorders}
n_1=(1+p)(1+p^2-s_p)\qquad \text{or} \qquad n_2=(1+p)(1+p^2-s_p+3p)\,.
\end{equation}
If $\lambda$ fails to divide one among $n_1$ and $n_2$, then the order $n$ is uniquely determined. Thus, we suppose from now on that $\lambda$ divides both $n_1$ and $n_2$, and hence that it divides their difference $3p(p+1)$. 
By \cite[Prop. 5.78]{CFADLNV05}, we have that
\begin{equation}\label{equation: Jacdec}
\Jac(C_p)(\F_p)\simeq \bigoplus_{j=1}^6 \Z/m_j\Z\,,
\end{equation}
where the integers $m_j$ satisfy $m_j\mid m_{j+1}$ for $1\leq j \leq 5$, $m_j\mid p-1$ for $1\leq j\leq 3$, and $m_6=\lambda$. 

Suppose first that $p$ does not divide $m_6$. Then $m_6$ must divide $3(1+p)$. Since $p\equiv 2 \pmod 3$ and for $1\leq j\leq 3$ we have that $m_j$ divides $p-1$ and $m_6$, this implies that $m_j\mid 2$ for $1\leq j\leq 3$. For $4\leq j\leq 6$, let us write $m_j=3(1+p)/c_j$ for some integer $c_j$. From \eqref{equation: twoorders} and the bounds on $s_p$, we obtain 
$$
(p+1)(p-1)^2\leq \#\Jac(C_p)(\F_p) \leq (p+1)^3\,.
$$
This implies
$$
1\leq\frac{c_4c_5c_6}{27 m_1m_2m_3} \leq \left(\frac{p+1}{p-1}\right)^2=:B(p)\,.
$$
A straightforward computation shows that $B(p)<1+2^{-3}\cdot 3^{-3}$ under the assumption of the statement that $p\geq 877$. Since the denominator of the central term of the above inequality is bounded by $2^3\cdot 3^3$, we deduce that $c_4c_5c_6=27m_1m_2m_3$, or equivalently that $n=(1+p)^3$, $s_p=p$, and $t_p=-2p$.

Suppose next that $p$ divides $m_6$. Then \eqref{equation: twoorders} implies that $s_p=p+1$. Therefore 
\begin{equation}\label{equation: n1n2mpdiv6}
n_1=(1+p)p(p-1)\qquad\text{and}\qquad n_2=(1+p)p(p+2)\,.
\end{equation}
We claim that this can only occur if $p$ is a Fermat prime. Indeed, suppose the contrary for the sake of contradiction. Then there is an odd prime $\ell$ dividing $p-1$. Note that such an $\ell$ does not divide $(1+p)p$. As $p\equiv 2 \pmod 3$, in particular we have that $\ell\not =3$ and thus $\ell$ does not divide $n_2$. Let $\ell'$ be a prime dividing $p+2$. Since $p\equiv 2 \pmod 3$, necessarily $\ell'\geq 5$. Note that such an $\ell'$ does not divide~$n_1$. By Lagrange's theorem, exactly one among $\ell$ and $\ell'$ divides $\lambda$. But this is a contradiction with the fact that $\lambda$ divides both~$n_1$ and~$n_2$. Thus there exists an integer $r$ such that $p=2^r+1$. If $r\geq 6$ (which is clearly satisfied by assumption), we have $v_2(n_1)\geq 7$, where $v_2$ denotes the $2$-adic valuation. Hence $n=n_1$ implies that $v_2(m_6)\geq 2$. On the other hand, we have that $v_2(m_6)=1$ if $n=n_2$.  
\end{proof}

In \S 4.1, the above theorem will be used to sketch a randomized algorithm of Las Vegas type with the same complexity as the deterministic one to compute the $L_p(C,T)$ for $p$ in $\fI(C)$.

\subsection{Interlude on a theorem of Sawin}\label{section: Sawin}

Throughout this section, let $A$ be an abelian variety defined over a number field $k$ and of dimension~$g$, let $E$ be a number field of degree $e$, and suppose that there exists a $\Q$-algebra homomorphism
$$
E\hookrightarrow \End(A)\otimes \Q\,,
$$
where $\End(A)$ denotes the ring of endomorphisms of $A$ defined over $k$.

By \cite[Thm. 1]{Fit20}, if $g=3$ and $E$ is imaginary quadratic, there exists a positive density set of primes of $k$ of ordinary good reduction for $A$. In Section~\ref{section: densityquestions}, we will  refine this result in the particular case that $A$ is the Jacobian of a \emph{generic} Picard curve~$C$. More precisely, we will show that every prime in $\fS(C)$ outside a density $0$ set is ordinary for $C$. 

The main input in the proof of \cite[Thm.~1]{Fit20} is a result due to Ogus (see for example \cite[Prop. 8, Prop. 10]{Fit20}), which allows to prove the existence of a positive density of ordinary primes by studying the $p$-divisibility of the trace of a certain Galois representation. In order to obtain the sought refinement of \cite[Thm.~1]{Fit20}, we will take the powerful approach of \cite{Saw16}, which by a finer argument permits to in fact compute the exact value of this density. In this section, we present a mild generalization of \cite[Thm.~1]{Saw16}. It is this generalization that will be employed in Section \ref{section: densityquestions}. 

Let $\ell$ be a rational prime, $V_\ell(A)$ denote the rational $\ell$-adic Tate module of~$A$, and  
$$
\varrho_{A,\ell}\colon G_k\rightarrow \Aut(V_\ell(A))
$$
be the associated $\ell$-adic representation. Our first task will be to define some spaces $V_\lambda(A)$, which are analogues of the spaces $H_\lambda^1(\Cbar)$ introduced in \eqref{equation: spacefirst} in our current more general setting. 
From now on, suppose that $\ell$ is totally split in $E$ so that $E_\lambda \simeq \Q_\ell$. For every prime $\lambda$ of $E$ lying above~$\ell$, let $E_\lambda$ be the completion of~$E$ at~$\lambda$, and let $V_\lambda(A)$ denote the tensor product $V_\ell(A)\otimes_{E\otimes \Q_\ell,\sigma_\lambda} E_\lambda$ taken with respect to the $\Q_\ell$-algebra map
$$
\sigma:=\sigma_\lambda \colon E\otimes \Q_\ell\simeq \bigoplus_{\lambda'|\ell}E_{\lambda'}\xrightarrow{\pr_\lambda} \Q_\ell\,,
$$
where $\pr_\lambda$ denotes the projection from the $E_\lambda$-component. Denote by $G_k$ the absolute Galois group of $k$. By letting $G_k$ act naturally on $V_\ell(A)$ and trivially on $E_\lambda$, one obtains a continuous representation
$$
\varrho_{A,\lambda}\colon G_k\rightarrow \Aut(V_\lambda(A))\simeq \GL_{2g/e}(\Q_\ell)
$$ 
unramified outside a finite set $S$ of primes of $k$ and which is integral and of weight~1 (by \cite[Thm. 2.11]{Rib76}). By integral, we mean that the characteristic polynomial of $\varrho_{A,\lambda}(\Frob_\p)$ has coefficients in the ring of integers $\cO_E$ of $E$ for every prime~$\p$ of $k$ outside~$S$. After fixing a polarization on $A$, the image $\Gamma$ of $\varrho_{A,\ell}$ sits inside $\GSp_{2g}(\Q_\ell)$. Let~$G$ denote the Zariski closure of $\Gamma$ inside $\GSp_{2g}$ (seen as an algebraic group over~$\Q_\ell$).
Let~$\Gamma_\lambda$ denote the image of $\varrho_{A,\lambda}$ inside $\GL_{2g/e}(\Q_\ell)$ and let $G_\lambda$ denote the Zariski closure of $\Gamma_\lambda$ inside $\GL_{2g/e}$. The isomorphism 
$$
V_\ell(A)\simeq \bigoplus_{\lambda} V_\lambda(A)
$$
of $\Q_\ell[G_k]$-modules induces a monomorphism $G\hookrightarrow \prod_{\lambda} G_\lambda$. 
Composing this injection with the projection to the $\lambda$-component yields a representation of $G$ inside $\GL_{2g/e}(\Q_\ell)$ that we will denote by $W_\lambda$. Let $\chi$ denote the restriction to $G$ of the similitude character of $\GSp_{2g}$. Let $w_\lambda$ denote the number of connected components of $G$ on which the trace of the representation $\wedge^2 W_\lambda \otimes \chi^{-1}$ is identically equal to a constant function. By $\Nm(\cdot)$ we denote both the absolute norm function on ideals of a given number field, and the absolute norm function on algebraic numbers. For a prime $\p$ outside $S$, let us denote by $b_{\p,\lambda}$ the trace\footnote{In the particular case that $A$ is the Jacobian of a Picard curve and $E=\Q(\zeta_3)$, this coincides with the definition of $b_{\p,\lambda}$ given in Section \ref{section: splitcase}.} $\Tr(\wedge^2\varrho_{A,\lambda}(\Frob_\p))$. Note that if $\wedge^2 W_\lambda \otimes \chi^{-1}$ is identically equal to the constant function $t\in \Q_\ell$ on a connected component of $G$, then there exist at least two primes $\p$ and $\q$ of $k$ such that
$$
t=\frac{b_{\p,\lambda}}{\Nm(\p)}=\frac{b_{\q,\lambda}}{\Nm(\q)}\,.
$$
This implies that $t$ in fact belongs to $\cO_E\subseteq \Q_\ell$. We will denote by $p$ the residue characteristic of the prime $\p$.

\begin{theorem}[After Sawin]\label{theorem: Sawin}
Suppose that $E$ is either $\Q$ or an imaginary quadratic field. The density of the set of primes $\p$ of $k$ outside $S$ such that $b_{\p,\lambda}$ is divisible by $p$ equals~$w_\lambda$ divided by the number of connected components of $G$.
\end{theorem} 

\begin{proof}
Let $T$ denote the set of elements $t$ in $\cO_E$ such that 
$|\Nm(t)|\leq \binom{2g/e}{2}^2$. For any real number $C\geq 0$ the set of elements~$a$ of~$\cO_E$ whose norm satisfies $|\Nm(a)|\leq C$ is finite. Indeed, if $E=\Q$ this is obvious, and if~$E$ is an imaginary quadratic field, then this follows from the fact that both the number of units in~$\cO_E$ and the number of ideals of absolute norm $\leq C$ are finite. Therefore the set~$T$ is finite. Since $\wedge^2\varrho_{A,\lambda}$ is of weight~$2$, for every prime $\p$ outside $S$ we have
$$
\left|\Nm\left(\frac{b_{\p,\lambda}}{p}\right)\right|\leq \dim(\wedge^2W_\lambda)^2=\binom{2g/e}{2}^2\,,
$$
and thus if $b_{\p,\lambda}=pt$, for some $t\in \cO_E$ and some $\p$ outside $S$, then $t\in T$. Let~$Z_\lambda^t$ denote the closed subset of $G$ on which $\wedge^2W_\lambda \otimes \chi^{-1}$ is identically equal to the constant function $t$. Let $Z_\lambda$ denote the finite union
$$
Z_\lambda:=\bigcup_{t\in T} Z_\lambda^t\,.
$$ 
Let $\mu$ denote the Haar measure of $\Gamma$, normalized so that it has total mass~$1$. Then exactly as in the proof of \cite[Thm. 1]{Saw16} one shows that $\mu(Z_\lambda\cap \Gamma)$ is the sought for density and that $\mu(Z_\lambda\cap \Gamma)$ equals $w_\lambda$ divided by the number of connected components of $G$.
\end{proof}

\subsection{Ordinary primes for generic Picard curves}\label{section: densityquestions}

Resume the notations from Sections \ref{section: splitcase} and \ref{section: inertcase}. In particular, let $C$ be a Picard curve defined over $\Q$. 
 
\begin{corollary}\label{corollary: genericordinary}
Suppose that $C$ is a generic Picard curve. Then, every prime~$p$ in $\fS(C)$ outside a density $0$ set is ordinary for $C$. 
\end{corollary}

\begin{proof}
Let $k=E=\Q(\zeta_3)$ and $A=\Jac(C)_k$. Denote by $G$ the Zariski closure of the image of the $\ell$-adic representation attached to $A$. By \cite[Prop. 2.17]{FKRS12}, when $C$ is generic, $G$ is connected (in fact, Upton \cite{Upt09} has shown that if $C$ is generic, then for all but finitely many $\ell\equiv 1 \pmod 3$ one has $\varrho_{A,\ell}(G_k)\simeq \GL_3(\Z_\ell)$). Since the trace of $\wedge^2 W_\lambda \otimes \chi^{-1}$ is not constant on $G$, the set of primes $p$ in $\fS(C)$ for which $b_{\p,\lambda}$ is divisible by $p$ has density $0$. But, by Lemma \ref{lemma: bpord} and Remark~\ref{remark: analogue}, a prime $p$ in $\fS(C)$ is ordinary for $C$ if and only if $b_{\p,\lambda}$ is not divisible by $p$. 
\end{proof}

Combining the above corollary with Corollary \ref{corollary: CartierManin}, we obtain the following result.

\begin{corollary}
Suppose that $C$ is a generic Picard curve. Then, for every prime~$p$ in $\fS(C)$ outside a density 0 set, the Cartier--Manin matrix of $C$ at $p$ uniquely determines the $L$-polynomial $L_p(C,T)$. 
\end{corollary}

By Lemma \ref{lemma: pmod3eq2}, no prime in $\fI(C)$ is ordinary. Let $b_p$ denote the coefficient of~$T^2$ in $L_p(C,T)$. For the sake of completeness, we study the density of the set of primes in $\fI(C)$ for which $b_p\equiv 0 \pmod p$.

\begin{corollary}\label{corollary: genericothercomp}
Suppose that $C$ is a generic Picard curve. Then, the set of primes~$p$ in $\fI(C)$ such that $b_p\equiv 0 \pmod p$ has density $0$. 
\end{corollary}

\begin{proof}
Since $C$ is generic, the Zariski closure $G$ of the image of $\varrho_{A,\ell}$ consists of two connected components $G^0$ and $G^1$, which are the respective Zariski closures of the images of $G_{\Q(\zeta_3)}$ and $G_\Q- G_{\Q(\zeta_3)}$. By Theorem \ref{theorem: Sawin} applied to $k=E=\Q$, $A=\Jac(C)$, and $\lambda=\ell$, it suffices to show that the trace of $\wedge^2W_\ell\otimes \chi$ on $G^1$ is not a constant function. 

Let $\ST(A)$ denote the Sato--Tate group of $A$. It is a compact real Lie subgroup of the unitary symplectic group $\USp(6)$ of degree $6$. Let $\U(3)$ denote the unitary group of degree $3$ in its standard representation, and let $I_3\in \U(3)$ denote the identity matrix. Let $W$ denote the standard representation of $\USp(6)$. By \cite[\S3.3.1]{FKS19}, we have that $\ST(A)=\langle \ST(A)^0, J\rangle$, where
$$
\ST(A)^0=\left\{ 
\begin{pmatrix}
u & 0\\
0 & \overline u
\end{pmatrix}: u\in \U(3)
\right\}
\quad
\text{and}
\quad
J=
\begin{pmatrix}
0 & I_3\\
-I_3 & 0
\end{pmatrix}
\,.
$$ 
As argued in the proof of \cite[Thm. 3]{Saw16}, the trace of $\wedge^2W_\ell\otimes \chi$ on~$G^1$ is a constant function if and only if the trace of $\wedge^2W$ on $J\ST(A)^0$ is a constant function. But the latter is not true, since if we set
$$
A=J\begin{pmatrix}
u & 0\\
0 & \overline u
\end{pmatrix}\,,
\quad\text{with }
u=\begin{pmatrix}
0 & 0 & 1\\
1 & 0 & 0\\
0 & 1 & 0
\end{pmatrix}\,,
$$
then $\Tr\wedge^2W(A)=0$, while $\Tr\wedge^2W(J)=3$.
\end{proof}

\section{\texorpdfstring{$L$}{L}-polynomials of Picard curves: a practical algorithm}\label{section: algorithmic}

Let $C$ be a generic Picard curve defined over $\Q$. In this section we use the results obtained in Section \ref{section: theoretical} to develop and implement an algorithm for the computation of the $L$-polynomials $L_p(C,T)$ for almost all primes up to some bound $N$. We first describe the algorithm, then analyze its running time and correctness, and finally discuss its implementation.

\subsection{Description}\label{section: algdesc} Denote by $\fSst(C)$ the subset of ordinary primes in $\fS(C)$. 
For $N\geq 1$, set
$$
\fSst_N(C):=\fSst(C)\cap [1,N]\,,\qquad \fI_N(C):=\fI(C)\cap [1,N]\,.
$$

We will rely on the following algorithms documented in the literature:

\begin{itemize}
\item \textsc{ComputeCartierManinMatrices} (\cite[p. 10]{Sut20}). Algorithm that, given an integer $N\geq 1$ and a Picard curve $C$, computes the Cartier--Manin matrix $A_p$ of $C$ at $p$ for every prime $5 \leq p\leq N$ not dividing the discriminant of $f$.\\ 
Running time: $N\log(N)^{3+o(1)}$.\vspace{0,2cm}
\item \textsc{FindCubicRoots} (See Remark \ref{remark: findcubicroot}).
Algorithm that, given an integer $N\geq 1$, returns a primitive cubic root of unity in~$\F_p^\times$ for every prime $p\equiv 1\pmod 3$ and~$\leq N$.\\
Running time: $N\log(N)^{3+o(1)}$.\vspace{0,2cm}
\item \textsc{IsIrreducible} (\cite[Thm. 14.4]{GG13}). Algorithm that, given the polynomial $\tilde f\in \F_p[x]$ of degree~$4$, determines whether it is irreducible or not.\\
Running time: $\log(p)^{2+o(1)}$. \vspace{0,2cm}
\item \textsc{HasRationalRoot} (See the Corollary in the Appendix). Algorithm that, given the polynomial $\tilde \psi_f\in \F_p[x]$ of degree~$9$, determines whether it has an $\F_p$-rational root or not.\\
Running time: $\log(p)^{2+o(1)}$. \vspace{0,2cm}
\end{itemize}

\begin{remark}\label{remark: findcubicroot}
We describe algorithm \textsc{FindCubicRoots}. Pick an elliptic curve with CM discriminant $D=-12$, for example $E\colon y^2 = x^3 - 15x + 22$. The algorithms given in \cite{HS14}, \cite{HS16}, and \cite{Sut20} compute the trace of Frobenius $t_p$ of $E$ at $p$ for all primes $p\leq N$ in time $N\log(N)^{3+o(1)}$, excluding a finite subset of all primes. For every prime $p\equiv 1 \pmod 3$, which is necessarily of good ordinary reduction, there exists an integer $u_p$ satisfying the norm equation 
$$
4p = t_p^2 + 12u_p^2\,.
$$
Since $p$ is ordinary for $E$, $u_p$ is invertible modulo $p$, and the above equation provides a way to compute a square root of $-3$ in $\F_p$ as $t_p/(2u_p)$, and hence $\zeta_3 \in \F_p$. We note that this is a one-time computation if one wishes to run the algorithm for multiple curves.
\end{remark}

At this point we have described all the theoretical and computational tools necessary in order to state the main algorithm of this article.

\begin{algorithm}[\textsc{ComputeLpolynomials}]\label{algorithm: full} 
Given an integer $N\geq 1$ and a Picard curve $C$ over $\Q$, compute $L_p(C,T)$ for every prime $p\in \fSst_N(C)\cup\fI_N(C)$, by following the steps:
\begin{enumerate}[a)]
\item Apply \textsc{ComputeCartierManinMatrices} to obtain $A_p$ for every prime $5 \leq p\leq N$ not dividing the discriminant of $f$.
\item By means of \textsc{FindCubicRoots}, find primitive cubic roots of unity for every $p\equiv 1 \pmod 3$ up to $N$.
\item For every $p\equiv 1 \pmod 3$ as in $a)$, determine whether $p\in \fSst(C)$, by computing $\rk(A_p)$, and if so then:
\begin{enumerate}[1)]
\item If $p< 53$, then use naive point counting to compute $L_p(C,T)$. Otherwise continue to $2)$.
\item Let $\sigma(\zeta_3),\sigmabar(\zeta_3)\in \F_p^{\times}$ be the primitive cubic roots of unity found in b). Solve the linear system \eqref{equation: linear system} to determine $a_\p$ (and hence $a_{\overline{\p}}$).
\item Choose an integer $1\leq\gamma\leq p$ such that $\gamma\equiv \sigma(\zeta_3)\pmod p$. Determine $\pi\in \Z[\zeta_3]$ as the $\gcd(\gamma-\zeta_3,p)$, by applying the extended Euclidean division algorithm in $\Z[\zeta_3]$.
\item Solve the linear equation \eqref{equation: determinezeta} to determine $\zeta$ and thereby $c_\p$.
\item Apply part $i)$ of Lemma \ref{lemma: bpord} to determine $b_\p$ and thereby $L_p(C,T)$.
\end{enumerate}
\item For every $p\equiv 2 \pmod 3$ as in $a)$, let $t_p$ be as in Lemma \ref{lemma: pmod3eq2}. Then:
\begin{enumerate}[1)]
\item Determine $t_p$ modulo $p$ by means of $A_p$.
\item Apply the Corollary in the Appendix to determine $t_p$ modulo $2$ by applying \textsc{HasRationalRoot} to $\tilde\psi_f$.
\item Apply Lemma~\ref{lemma: fmodp} or Proposition~\ref{proposition: Lpolymod3} to determine $t_p$ modulo $3$ by applying \textsc{IsIrreducible} to $\tilde f$. 
\item Determine $L_p(C,T)$ from the above information by using Lemma \ref{lemma: pmod3eq2}.
\end{enumerate}
\end{enumerate}
\end{algorithm}

\begin{theorem} Let $C$ be a Picard curve. Algorithm \ref{algorithm: full} computes $L_p(C,T)$ for every prime $p\in \fSst_N(C)\cup\fI_N(C)$ in time $N \log(N)^{3+o(1)}$. When $C$ is generic, the complement of $\fSst(C)\cup\fI(C)$ in the set of primes has density $0$.
\end{theorem}

\begin{proof} Once we have computed the Cartier--Manin matrix by using \textsc{ComputeCartierManinMatrices}, correctness of the algorithm follows from the proofs of Corollary~\ref{corollary: CartierManin} and Theorem~\ref{theorem: inertLpolys}. Note that in step c) 3), we may set $\pi=\gcd(\gamma-\zeta_3,p)$ in virtue of the Dedekind--Kummer theorem. 
As for complexity, note that both step a) and b) already have the claimed complexity (although we note that b) is a one-time computation). We only need to show that the remaining steps do not exceed this complexity. As we have seen, none of the algorithms \textsc{IsIrreducible} and \textsc{HasRationalRoot} does, and the remaining steps are clearly faster as they solve a small linear system or compute a small gcd in $\Z[\zeta_3]$. In the generic case, the claim about density is ensured by Corollary~\ref{corollary: genericordinary}.
\end{proof}

\begin{remark}\label{remark: genericity}
There exist several methods of verifying that a given Picard curve $C$ defined over $k=\Q(\zeta_3)$ is generic. On the one hand, one may apply the criterions provided by Zarhin (specifically see \cite[Thm 1.3]{Zar18}). On the other hand, one can use results due to Upton. More precisely, we claim that if $\ell\geq 5$ is a prime $\equiv 1 \pmod 3$ such that the mod-$\ell$ image of the Galois representation associated to $A:=\Jac(C)$ is $\GL_3(\F_\ell)$, then $C$ is generic. Indeed, by \cite[Prop. 6]{Upt09}, the hypotheses imply that the image of the $\ell$-adic representation $\varrho_{A,\ell}$ is $\GL_3(\Z_\ell)$. Then, for any finite extension $L/k$, the commutant of $\varrho_{A,\ell}(G_L)$ in $\End(V_\ell(A))$ is $2$-dimensional. By Faltings' isogeny theorem, this commutant is $\End(A_L)\otimes \Q_\ell$, and hence $\End(A_\Qbar)\simeq \Z[\zeta_3]$, which completes the proof of the claim. Verifying that the mod-$\ell$ image of the Galois representation associated to $A$ is $\GL_3(\F_\ell)$ can be done by using \cite[Lemma 3]{Upt09} and computing a few $L_p(C,T)$ (this is done in \cite[\S6]{Upt09} for two specific curves, to which we will return in \S\ref{section: implementation}).
\end{remark}

\begin{remark}\label{remark: especulation1}
Note that by Dirichlet's density theorem and by Corollary \ref{corollary: genericordinary}, the set of primes $p\leq N$ in $\fS(C)-\fSst(C)$ has size $o(N/\log(N))$. Despite being unable to prove it, it is conceivable that this set is actually small enough so that the combination of Algorithm~\ref{algorithm: full} with the algorithm of \cite{ABCMT19} (to separately treat the primes in $\fS(C)-\fSst(C)$) would yield an algorithm to compute $L_p(C,T)$ for every prime $p\leq N$ in time $N\log(N)^{3+o(1)}$.  
\end{remark}

We will now sketch how one can use the proof of Theorem \ref{theorem: inertLpolysJacobian} to give a randomized algorithm of Las Vegas time accomplishing the same task as \textsc{ComputeLPolynomials} and having the same expected complexity (for details and an actual implementation we refer to \cite{AP20}). 

While from a computational perspective the determination of the exponent $\lambda$ of $\Jac(C_p)(\F_p)$ presents difficulties, a detailed examination of the proof of Theorem~\ref{theorem: inertLpolysJacobian} shows that only partial (and effectively computable) information on $\lambda$ is required in order to uniquely determine $L_p(C,T)$.

The variant that we want to discuss only differs from Algorithm \ref{algorithm: full} in case $p$ is in $\fI(C)$. Let $p\geq 877$ be one such prime. Assume given the data $i)$ and $ii)$ of Theorem~\ref{theorem: inertLpolysJacobian}. Suppose also that we are in the case that there exists $p\leq s_p\leq 2p$ such that $t_p\equiv s_p \pmod {3p}$, and let $n_1$ and $n_2$ be as defined by \eqref{equation: twoorders}. We claim that exactly one of the following options occurs:
\begin{enumerate}[i)]
\item $s_p\not =p,p+1$ or $s_p=p+1$ and $p$ is not a Fermat prime. In this case, $\lambda \nmid 3p(p+1)=n_2-n_1$ so it cannot divide both $n_1$ and $n_2$. Thus, there is a point $P$ in $\Jac(C_p)(\F_p)$ such that $n_i\cdot P\not =0$ for $i=1$ or $2$, and then $n=n_{3-i}$.
\item $s_p = p$ and there is a prime $\ell\geq 5$ dividing $\lambda$ but not $p+1$. In this case $n=(p+1)(p^2-p+1)$.
\item $s_p = p$ and there are three points $P,\, Q, R$ in $\Jac(C_p)(\F_p)$ generating a subgroup of order divisible by $3^{v_3(p+1)+2}$. In this case $n=(p+1)^3$.
\item $s_p=p+1$, $p$ is a Fermat prime, and there is an odd prime $\ell$ dividing $\lambda$ but not $p(p+1)(p-1)$. In this case $n=(1+p)p(p+2)$.
\item $s_p=p+1$, $p$ is a Fermat prime, and there is a point $P$ in $\Jac(C_p)(\F_p)$ of order divisible by $4$. In this case $n=(1+p)p(p-1)$.
\end{enumerate}

Indeed, it has been shown in the course of the proof of Theorem \ref{theorem: inertLpolysJacobian} that we have the conclusion of $i)$ if $s_p\not=p,p+1$ or if $s_p=p+1$ and $p$ is not a Fermat prime. 
From \eqref{equation: twoorders},  $s_p=p$ implies that $n$ is either $n_1=(p+1)(p^2-p+1)$ or $n_2=(p+1)^3$. To show that $ii)$ and $iii)$ are mutually exclusive, it suffices to note that $v_3(p^2-p+1)=1$. Thus, $v_3(n_1)=v_3(p+1)+1$ and $p^2-p+1$ has a divisor $\ell\geq 5$. Such an $\ell$ necessarily fails to divide $p+1$. Recall the values $m_j$ introduced in \eqref{equation: Jacdec}. It remains to show that if $n=n_2$, then $v_3(m_4m_5m_6)\geq v_3(p+1)+2$. But we have seen that $v_3(m_j)=0$ for $1\leq j\leq 3$ and therefore we have $v_3(m_4\cdot m_5\cdot m_6)= 3v_3(p+1)\geq v_3(p+1)+2$. 
From \eqref{equation: twoorders}, $s_p=p+1$ implies that $n$ is either $n_1=(1+p)p(p-1)$ or $n_2=(1+p)p(p+2)$. Under the assumption that $p$ is a Fermat prime, the conclusions of $iv)$ and $v)$ have been seen at the end of the proof of Theorem \ref{theorem: inertLpolysJacobian}. 

Given these five cases, we may produce a Las Vegas algorithm by generating random points of the Jacobian (using \cite[\S12.2]{MJS20}, for example) and checking these conditions efficiently (using the methods of \cite{Sut11}, for example). We note that to check each such condition only requires expected $O(1)$ random elements of $\Jac(C_p)(\F_p)$. On a generic Picard curve, \S 3.4 ensures a dense subset of primes fall into case $i)$, the simplest case to check. The expected complexity of such an algorithm is identical to that of the deterministic algorithm, but in practice slower.

\subsection{Implementation} \label{section: implementation}

The Github repository \cite{AP20} includes an implementation of Algorithm \ref{algorithm: full} in \cite{PARIGP}. This repository also contains an implementation in \cite{Sage} of the Las Vegas algorithm discussed at the end of the previous section, for which we used the Jacobian arithmetic from \cite{MJS20}. We note that this algorithm is still practical, but is slower by a constant factor due to Jacobian arithmetic being expensive. 

In either case, the full algorithm for lifting is significantly faster than previous algorithms and has complexity $N\log(N)^{3+o(1)}$ when computing $Z_p(C,T)$ for almost all $p\le N$. Our algorithm can also operate independently for fixed $p$ when $L_p(C,T) \pmod{p}$ is already provided - this means we can speed up algorithms such as \cite{ABCMT19} for a specific prime, since the most expensive lifting component is now negligible. As a result, it can also be used to more efficiently compute $L_p(C,T)$ when $p$ is large enough that computing $L_p(C,T)$ for almost all $p\le N$ is impractical. For a given generic Picard curve over $\Q$, in Table \ref{tab: fullruntimes} we report the average running time per prime over the primes $\le N$ of Algorithm \ref{algorithm: full} with the implementation in \cite{PARIGP}.

\begin{table}[!ht]
\begin{adjustwidth}{-4em}{-4em}
	\begin{center}
		\begin{tabular}{lllllllll}
		\toprule
		\multicolumn{1}{l|}{$N$}                 & \multicolumn{2}{c|}{\centering{$2^{16}$}}                           & \multicolumn{2}{c|}{\centering{$2^{20}$}}                                   & \multicolumn{2}{c|}{$2^{24}$}                                    & \multicolumn{2}{c}{$2^{28}$}               \\ \hline
		\multicolumn{1}{l|}{Algorithm}           & \multicolumn{1}{l|}{Old} & \multicolumn{1}{l|}{New} & \multicolumn{1}{l|}{Old}  & \multicolumn{1}{l|}{New}        & \multicolumn{1}{l|}{Old}   & \multicolumn{1}{l|}{New}        & \multicolumn{1}{l|}{Old}   & New       \\ \hline
		\rowcolor{gray!20} 
\multicolumn{9}{l}{Full zeta function} \\ \hline
		\multicolumn{1}{l|}{$C_1$}       & \multicolumn{1}{l|}{113.2}     & \multicolumn{1}{l|}{\textbf{0.31}}       & \multicolumn{1}{l|}{215.1} & \multicolumn{1}{l|}{\textbf{0.69}} & \multicolumn{1}{l|}{1152.7} & \multicolumn{1}{l|}{\textbf{1.50}} & \multicolumn{1}{l|}{5051.4} & \textbf{4.76} \\ \cline{1-1}
		\multicolumn{1}{l|}{$C_2$} & \multicolumn{1}{l|}{111.3}     & \multicolumn{1}{l|}{\textbf{0.33}}       & \multicolumn{1}{l|}{213.5} & \multicolumn{1}{l|}{\textbf{0.71}} & \multicolumn{1}{l|}{1152.9} & \multicolumn{1}{l|}{\textbf{1.54}} & \multicolumn{1}{l|}{5053.9} & \textbf{4.87} \\ \hline
		\rowcolor{gray!20} 
\multicolumn{9}{l}{Lifting component} \\ \hline
		\multicolumn{1}{l|}{$C_1$}       & \multicolumn{1}{l|}{}     & \multicolumn{1}{l|}{\textbf{0.12}}       & \multicolumn{1}{l|}{}      & \multicolumn{1}{l|}{\textbf{0.12}} & \multicolumn{1}{l|}{}       & \multicolumn{1}{l|}{\textbf{0.13}} & \multicolumn{1}{l|}{}       & \textbf{0.14} \\ \cline{1-1}
		\multicolumn{1}{l|}{$C_2$} & \multicolumn{1}{l|}{}     & \multicolumn{1}{l|}{\textbf{0.12}}       & \multicolumn{1}{l|}{}      & \multicolumn{1}{l|}{\textbf{0.12}} & \multicolumn{1}{l|}{}       & \multicolumn{1}{l|}{\textbf{0.13}} & \multicolumn{1}{l|}{}       & \textbf{0.14} \\ \hline
		\rowcolor{gray!20}
\multicolumn{9}{l}{Cartier--Manin matrix} \\ \hline
		\multicolumn{1}{l|}{$C_1$}       & \multicolumn{1}{l|}{35.4}     & \multicolumn{1}{l|}{0.19}       & \multicolumn{1}{l|}{79.1}      & \multicolumn{1}{l|}{0.57}          & \multicolumn{1}{l|}{277.8}       & \multicolumn{1}{l|}{1.37}          & \multicolumn{1}{l|}{1368.3}       & 4.63          \\ \cline{1-1}
		\multicolumn{1}{l|}{$C_2$} & \multicolumn{1}{l|}{34.4}     & \multicolumn{1}{l|}{0.20}       & \multicolumn{1}{l|}{77.6}      & \multicolumn{1}{l|}{0.59}          & \multicolumn{1}{l|}{276.7}       & \multicolumn{1}{l|}{1.41}          & \multicolumn{1}{l|}{1366.7}       & 4.74          \\ \bottomrule
		\end{tabular}
	\end{center}
	\bigskip
	\caption{For the curves $C_1: y^3=x^4+x+1$ and $C_2: y^3=x^4+3x^2+2x+1$, we display the average computation time spent per prime $p\le N$ for \cite{ABCMT19}, \cite{Sut20}, and Algorithm \ref{algorithm: full} in milliseconds (ms). Bold numbers indicate new results from our algorithm. The timings are taken on a 3.40GHz Intel(R) Xeon(R) E5-2687W CPU. All computations were run with a single core.}
	\label{tab: fullruntimes}
\end{adjustwidth}

\end{table}

In Table \ref{tab: fullruntimes}, to compute these running times we ran each computation for $p\le N$ three times and took the minimum of those attempts. 
The `New' columns have running times for our algorithm and \cite{Sut20}, since they are used together to compute the full zeta function. In the `Lifting component' section, we have put running times for Algorithm \ref{algorithm: full} and in the `Cartier--Manin matrix' section we have put running times for an implementation of \cite{Sut20} in C.
In the `Old' columns, we have the implementation of the algorithm \cite{ABCMT19} in \cite{Sage} computing the full zeta function. We have also used the Frobenius matrix method in \cite{Sage} to compute a matrix which determines the Cartier--Manin matrix. These times are shown in `Old' columns of the `Full zeta function' and `Cartier--Manin' matrix sections respectively in Table \ref{tab: fullruntimes}. Running times for all of these algorithms have been given for two chosen generic curves $C_1: y^3=x^4+x+1$ and $C_2: y^3=x^4+3x^2+2x+1$ (see \cite[\S 6.1.1, \S 6.2.2]{Upt09} for a proof of the genericity of these curves; we note that Zarhin's criterion recalled in Remark~\ref{remark: genericity} also applies). For example, we can read off the table that for $C_1$ computing the zeta function with our algorithm for $p\le 2^{28}$ takes an average of 4.76 ms per prime. Of this time, on average we spend 0.14 ms on the lifting component described in this paper and 4.63 ms on computing the Cartier--Manin matrices. In comparison, the implementation of \cite{ABCMT19} takes roughly 5051.4 ms per prime.

We note that \cite{ABCMT19} is not fast enough to compute the full zeta function for all $p\le N$ in a reasonable amount of time when $N$ is very large, so we sampled one out of $10^2, 10^3, 10^4, 10^5$ primes respectively for $N=2^{16},2^{20},2^{24}$ and $2^{28}$ to get estimates for the running times. The same was done for the computation of the Frobenius matrices. Each computation was run on a single core in order to have a fair comparison. Algorithm ~\ref{algorithm: full} is used in the `Lifting component' part of the table to compute the full zeta function once the algorithm from \cite{Sut20} finishes. As can be seen from the table, the lifting component is always faster on average and performs better relative to computing $L_p(C,T) \pmod{p}$ as $p$ becomes very large. There is some variation in the average running times between the two curves, but in each example this is $\approx 10^{-4}$ ms which is why it does not appear in the table.

As noted in the beginning of this section, Algorithm \ref{algorithm: full} also allows us to speed up \cite{ABCMT19} on a Picard curve for a specific prime $p$ by a constant factor. In particular, we need only compute $L_p(C,T) \pmod{p}$, and then lifting this takes a negligible amount of time. The time for the computating $L_p(C,T) \pmod{p}$ differs by a constant factor compared to computing $L_p(C,T)$ using \cite{ABCMT19}.
Using \cite{Sage}, we can estimate this constant factor to be about $8$, for example from
\begin{verbatim}
	sage: p=(2**40).next_prime()
	sage: x = PolynomialRing(GF(p),"x").gen()
	sage: CyclicCover(3, x^4 + x + 1).frobenius_matrix(1)
\end{verbatim}
which takes a total of 3 min and 23 s, while using the algorithm to compute the full result $L_p(C,T)$ takes 24 min and 24 s. By applying Algorithm \ref{algorithm: full} after computing the Frobenius matrix modulo $p$ we improve the running time by around a factor of 8, because for primes of this size our algorithm still takes on the order of a millisecond, which is negligible.

\subsection*{Acknowledgements} Thanks to Andrew Sutherland for suggesting the problem, for his guidance throughout the project, for triggering the collaboration between the authors, and for writing a delightful appendix crucial to this work. Thanks to Chun Hong Lo for valuable advice as well as help proofreading this paper. Thanks to Kiran Kedlaya and Bjorn Poonen for helpful conversations, and to the anonymous referees for their valuable suggestions and corrections. Fit\'e was financially supported by the Simons Foundation grant 550033.

\newpage

\section*{Appendix}
\begin{center}
{by Andrew V. Sutherland}\footnote{Department of Mathematics, Massachusetts Institute of Technology, 77 Massachusetts Ave., Cambridge, MA 02139, USA; email: \texttt{drew@math.mit.edu}, URL: \texttt{https://math.mit.edu/~drew}}
\end{center}

In this appendix we give a constructive proof of the following theorem.
\begin{theorem*}
Let $p>3$ be a prime congruent to $2$ modulo $3$, let $C_p\colon y^3=f(x)$ be a Picard curve over $\F_p$ with $L$-polynomial $L_p\in\Z[T]$ .  The reduction of $L_p(T)$ modulo $2$ can be computed (deterministically) in $O((\log p)^2(\log \log p))$ time.
\end{theorem*}

Without loss of generality we may assume $C_p$ is the reduction of a Picard curve $C\colon y^3=f(x)$ over $\Q$.  The Jacobian $J:= \Jac(C)$ is an abelian variety of dimension~$3$, thus the action of $\Gal(\Qbar/\Q)$ on $J[2]$ gives rise to a mod-2 Galois representation
\[
\bar\rho_2\colon \Gal(\Qbar/\Q)\to \Aut(J[2])\simeq \GSp_6(\F_2)=\Sp_6(\F_2),
\]
where we have used the Weil pairing to view $J[2]\simeq \F_2^6$ as a symplectic space. The endomorphism ring of the base change of $J$ to $\Q(\zeta_3)$ contains $\Z[\zeta_3]$ and the prime $2$ is inert in $\Q(\zeta_3)$; these facts imply that the restriction of $\bar\rho_2$ to $\Gal(\Qbar/\Q(\zeta_3))$ has image in $U(3,2)=\GU(3,\F_4)$; see \cite[\S 2]{Upt09} for details.  

Up to conjugacy, there is a unique subgroup $H\subseteq \Sp_6(\F_2)$ that is isomorphic to $\GU(3,\F_4)$, which has small group identifier $\langle 648,533\rangle$.\footnote{In this appendix isomorphism classes of groups of order $n<2048$ are identified by the label $\langle n,i\rangle$ assigned to them by the Small Groups Library \cite{BEO01} used in GAP, Sage, and Magma.}
The endomorphism ring of $J/\Q$ does not contain $\Z[\zeta_3]$, so the image of $\bar\rho_2$ does not lie in $H$, in general. It lies in a subgroup of $\Sp_6(\F_2)$ isomorphic to $\GU(3,\F_4)\rtimes\Gal(\F_4/\F_2)$, with small group identifier $\langle 1296,2891\rangle$, which contains $H$ with index 2.  Up to conjugacy there is a unique such subgroup $G$; it is the normalizer of $H$ in $\Sp_6(\F_2)$.  Let $I:=\left(\begin{smallmatrix}1&0\\0&1\end{smallmatrix}\right)$, $z_3:=\left(\begin{smallmatrix}1&1\\1&0\end{smallmatrix}\right)$, $s_2:=\left(\begin{smallmatrix}1&1\\0&1\end{smallmatrix}\right)$, and define
\[
A:=\begin{pmatrix} z_3 & 0 & 0\\ 0 & I & 0\\ 0 & 0 & z_3\end{pmatrix},\qquad B:=\begin{pmatrix} 0 & 0 & z_3^2\\0&z_3^2&z_3^2\\z_3^2&z_3^2&z_3\end{pmatrix},\qquad S_2:=\begin{pmatrix}s_2&0&0\\0&s_2&0\\0&0&s_2\end{pmatrix}.
\]
Then we can take $H=\langle A,B\rangle$ and $G=\langle A, B,S_2\rangle$.

For each prime $p>3$ of good reduction for $C$ the representation $\bar \rho_2$ maps the Frobenius element $\Frob_p\in \Gal(\Qbar/\Q)$ to a conjugacy class of $\Sp_6(\F_2)$ whose characteristic polynomial is the reduction of the characteristic polynomial $\chi_p\in \Z[T]$ of the Frobenius endomorphism $\pi$ of $J_p:=\Jac(C_p)$, equivalently, the characteristic polynomial of the endomorphism of $J_p[2]\simeq \F_2^6$ induced by $\pi$.  The polynomials $\chi_p(T)$ and $L_p(T)$ are reciprocal, meaning that $L_p(T)=T^6\chi_p(T^{-1})$, as are their reductions to $\F_2[T]$.  The image of $\Frob_p$ under $\bar\rho_2$ lies $H$ when $p$ splits in $\Q(\zeta_3)$, and in the complement $G-H$ when $p$ is inert in $\Q(\zeta_3)$, equivalently, when $p\equiv 2\pmod 3$.

Thus to compute the reduction $\bar L_p\in \F_2[T]$ for $p\equiv 2\pmod 3$, it suffices to determine the characteristic polynomial of $\bar\rho_2(\Frob_p)\in G-H$, for which there are only two possibilities:
\[
T^6+T^4+T^2+1\qquad\text{or}\qquad T^6+1.
\]
The first occurs whenever $\bar\rho_2(\Frob_p)$ has order 2 or 8; the second occurs when $\bar\rho_2(\Frob_p)$ has order 6.  It is an easy computation in \cite{Sage} to verify this fact, and that every element of $G-H$ has order 2, 6, or 8.

The fixed field of the kernel of $\bar\rho_2$ is the 2-torsion field $\Q(J[2])$, and we have an isomorphism $\Gal(\Q(J[2])/\Q)\simeq \im\bar\rho_2$.
It follows that to compute $\bar L_p\in \F_2[T]$ it suffices to determine the order of the restriction of $\Frob_p$ to $\Gal(\Q(J[2])/\Q)$.  The group~$G$ contains a unique normal subgroup $Z$ of order 3 generated by the matrix $Z_3:=\mathrm{diag}(z_3,z_3,z_3)$ which gives the action of $\zeta_3\in \End(J_{\Q(\zeta_3)})$ on $J[2]$.  The quotient $G/Z$ is isomorphic to the general affine group $\AGL(3,2)$, with small group identifier $\langle 432, 734\rangle$;  the projection $G\to G/Z$ does not change the order of elements of $G-H$.

The group $G/Z\simeq \mathrm{AGL}(3,2)$ has a natural permutation representation of degree~9 corresponding to the transitive group with LMFDB  label \href{https://www.lmfdb.org/GaloisGroup/9T26}{\texttt{9T26}} \cite{LMFDB}.  This permutation representation can be realized via the action of $\Gal(\Qbar/\Q)$ on the $\zeta_3$-orbits of the 27 affine bitangents of a generic Picard curve.

As shown in \cite[\S 2.2]{BTW05}, each point of order 2 in $J_p(\F_p)$ corresponds to the class of an ideal of the form $\langle s,u+y\rangle$ in the ring $\F_p[x,y]/(y^3-f(x))$, with $s,u\in \F_p[x]$ satisfying $\deg u < \deg s\le 3$, with $s$ monic and $s^2 = u^3+f$.  Those with $\deg s=2$ correspond to divisors of the form $P_1+P_2-2\infty$ arising from the 27 bitangents that intersect the curve at the affine points $P_1,P_2$, where $s(x)=(x-x(P_1))(x-x(P_2))$ and the linear polynomial $u$ is determined by $u(x(P_i))=y(P_i)$.  If we assume $f(x)=x^4+f_2x^2+f_1x+f_0$ and let $s(x)=x^2+s_1x+s_0$ and $u(x)=u_1x+u_0$, equating coefficients on both sides of the equation $s^2=u^3+f$ yields the system
\[
2s_1=u_1^3,\quad  2s_0 = f_2+3u_0u_1^2-s_1^2,\quad 2s_0s_1 = f_1+3u_0^2u_1,\quad s_0^2 = f_0+u_0^3.
\]
A Groebner basis calculation over $\Q(f_0,f_1,f_2)$ shows that $s_1$ must be a root of
\begin{align*}
\psi_f(x) := x^9 &+ 24f_2x^7 - 168f_1x^6 + (1080f_0 - 78f_2^2)x^5 + 336f_1f_2x^4\\
 &+ (1728f_0f_2 - 636f_1^2 + 80f_2^3)x^3 + (-864f_0f_1 - 168f_1f_2^2)x^2\\
 &+ (-432f_0^2 + 216f_0f_2^2 - 120f_1^2f_2 - 27f_2^4)x - 8f_1^3,
\end{align*}
and $s_0$ can be written as a polynomial in $s_1$ whose coefficients are rational functions of $f_0,f_1,f_2$.  Each of the 9 possibilities for $(s_0,s_1)$ gives rise to 3 possibilities for $(u_0,u_1)$ that differ only by a cube root of unity and comprise a single $\zeta_3$-orbit.  The 27 possible values of $u_1$ are the roots of $\psi_f(x^3/2)$, each of which determines rational values for $u_0,s_0,s_1$.  The discriminant of $\psi_f$ has the form $-2^{24}3^{27}D_f^2$ with $D_f\in \Z[f_0,f_1,f_2]$, a fact we will use in the proof of the lemma below.

This calculation of $\psi_f(x)$ is valid over any field $k$ whose characteristic is not 2 or 3, and we can define $\psi_f$ for any Picard curve $y^3=f(x)$ over $k$ by
putting $f$ in the form $x^4+f_2x^2+f_1x+f_0$ as follows: if $f= \sum f_ix^i\in k[x]$ with $f_4\ne 1$ then replace $f$ with $f_4^3f(x/f_4)$ and if $f_4=1$ and $f_3\ne 0$ replace $f$ with $f(x-f_3/4)$.

\begin{lemma*}
Let $C\colon y^3=f(x)$ be a Picard curve over $k$ with $\mathrm{char}(k) \ne 2,3$ and let $J:=\Jac(C)$.
Then $k(J[2])$ is the splitting field $K$ of $\psi_f(x^3/2)$ over $k$, and when $k=\F_p$ with $p\equiv 2\pmod 3$ this is also the splitting field $K'$ of $\psi_f(x)$.
\end{lemma*}
\begin{proof}
The points in $J[2]$ corresponding to bitangents are defined over $K$, so $J(K)[2]$ has order at least~27 and must be 32 or 64, since $J[2]\simeq (\Z/2\Z)^6$.  The three cuberoots of unity in $\bar k$ lie in $K$ (take ratios of appropriate roots of $\psi(x^3/2)$), so $\Z[\zeta_3]\subseteq \End(J_K)$.  The endomorphism $\zeta_3$ acts bijectively on the points of order 2 in $J(K)[2]$, with $\zeta_3$-orbits of size 3.  We cannot have $\#J(K)[2]=32$ because 3 does not divide 31. Thus $\#J(K)[2]=64$ and $K=k(J[2])$.

Now assume $k$ is a finite field $\F_p$ with $p>3$ and $p \equiv 2\pmod 3$. The discriminant $-2^{24}3^{27}D_f^2$ of $\psi_f$ is a square in its splitting field $K'$, thus $\F_p\subsetneq \F_p(\zeta_3)\subseteq K'$.  The action of the Frobenius endomorphism $\pi$ on $J[2]$ is given by a matrix $M$ in the group $G:=\langle A,B,S_2\rangle\subseteq \Sp_6(\F_2)$ defined above that does not lie in the index 2 subgroup $H:=\langle A,B\rangle$.
For any such $M$, its order $n$ is unchanged in the quotient $G/Z_3$ which gives the action of $\pi$ on the $\zeta_3$-orbits of $J[2]$.  This action is determined by (and has the same order as) the action of $\pi$ on the $\zeta_3$-orbits of the points in $J[2]$ arising from bitangents of $C$, which coincides with the action of the $p$-power Frobenius automorphism on the roots of $\psi_f(x)$.  The $p$-power Frobenius automorphisms of $\F_p(J[2])$ and~$K'$ thus have the same order~$n$, so $\F_p(J[2])=K'$.
\end{proof}

\begin{corollary*}
Let $y^3=f(x)$ be a Picard curve over $\F_p$ with $p\ne 2,p\equiv 2\pmod 3$.
Let $\bar L_p\in \F_2[T]$ be the reduction of its $L$-polynomial modulo $2$.  Then
\[
\bar L_p(T) =\begin{cases}
T^6+T^4+T^2+1 & \text{if }\psi_f(x)\text{ has an $\F_p$-rational root,}\\
T^6+1 & \text{if }\psi_f(x)\text{ has no $\F_p$-rational roots.}\\
\end{cases}  
\]
These cases can be distinguished in $O((\log p)^2\log\log p)$ time by computing the degree of $\gcd(x^p-x,\psi_f(x))$.
\end{corollary*}
\begin{proof}
Let $J:=\Jac(C)$.  Then $\bar L_p(T)$ is reciprocal to the characteristic polynomial of the Frobenius endomorphism $\pi$ as an element of $\End(J[2])$, corresponding to a matrix $M$ in the group $G:=\langle A,B,S_2\rangle\subseteq \Sp_6(\F_2)$ defined above that does not lie in the index two subgroup $H:=\langle A,B\rangle$.
As noted above, the two possibilities for $\bar L_p(T)$ are those listed in the statement of the corollary; the first occurs when $M$ has order 2 or 8, while the second occurs when $M$ has order~6.  The order of $M$ is equal to the order of $\Gal(\F_p(J[2])/\F_p)=\Gal(K'/\F_p)$, where $K'$ is the splitting field of $\psi_f(x)$ over $\F_p$.  The Galois group of $\psi_f(x)$ over $\F_p$ is a cyclic subgroup of the transitive group \href{https://www.lmfdb.org/GaloisGroup/9T26}{\texttt{9T26}} whose generator $\sigma$ does not lie in the unique index~2 subgroup (we note that this is true even if $\psi_f(X)$ is not squarefree).
The cycle structure of the degree 9 permutation $\sigma$ must be one of $2^31^3$, $8^11^1$, or $6^13^1$, the last of which has order 6 and occurs if and only if $\psi_f(x)$ has no $\F_p$-rational roots.

As is well known, the roots of $\gcd(x^p-x,\psi_f(x))$ are the distinct $\F_p$-rational roots of $\psi_f(x)$; the degree of this polynomial is nonzero if and only if $\psi_f(x)$ has an $\F_p$-rational root. To efficiently compute $\gcd(x^p-x,\psi_f(x))$ one computes $x^p$ in the ring $\F_p[x]/(\psi_f(x))$ via binary exponentiation to obtain a polynomial $g\in \F_p[x]$ of degree less than $\deg \psi_f=9$ and then computes $\gcd(g(x)-x,\psi_f(x))$.  This involves $O(\log p)$ ring operations in $\F_p[x]/(\psi_f(x))$, each of which can be computed using $O(1)$ ring operations and a Euclidean division in $\F_p[x]$, followed by a GCD computation on polynomials of degree $O(1)$, which requires $O(1)$ ring operations and Euclidean divisions in $\F_p[x]$ using the standard Euclidean algorithm.

Each Euclidean division in $\F_p[x]$ can be accomplished using $O(1)$ ring operations in $\F_p[x]$ via Newton iteration \cite[Theorem 9.6]{GG13}, and each ring operation in $\F_p[x]$ involving polynomials of degree $O(1)$ can be achieved using $O(1)$ ring operations in $\Z$ via Kronecker substitution and Euclidean division in $\Z$.
The entire computation of $\gcd(x^p-x,\psi_f(x))$, including the cost of deriving $\psi_f(x)$ from $f(x)$, reduces to $O(\log p)$ ring operations on integers with $n=O(\log p)$ bits.  Applying the $O(n \log n)$ bound for integer multiplication \cite{HvdH21} completes the proof.
\end{proof}

\end{document}